\documentclass[reqno]{amsart}
\usepackage{amsmath,amssymb,cite}
\usepackage[mathscr]{euscript}
\usepackage[toc,title,page]{appendix}

\usepackage{tikz}

\usepackage{xcolor}
\usepackage[notcite,notref]{showkeys}

\theoremstyle{plain}
\newtheorem{theorem}{Theorem}[section]
\newtheorem{proposition}[theorem]{Proposition}
\newtheorem{lemma}[theorem]{Lemma}
\newtheorem{corollary}[theorem]{Corollary}

\theoremstyle{definition}

\newtheorem{remark}[theorem]{Remark}
\newtheorem{example}[theorem]{Example}

\numberwithin{equation}{section}
\numberwithin{theorem}{section}

\def\be{\begin{equation}}
\def\ee{\end{equation}}
\def\bp{\begin{pmatrix}}
\def\ep{\end{pmatrix}}
\def\bea{\begin{eqnarray}}
\def\eea{\end{eqnarray}}

\def\\{\par\medskip}

\let\0=\noindent

\newcommand{\mc}[1]{{\mathcal #1}}
\newcommand{\mb}[1]{{\mathbf #1}}
\newcommand{\mf}[1]{{\mathfrak #1}}

\newcommand{\bb}[1]{{\mathbb #1}}
\newcommand{\ms}[1]{{\mathscr #1}}

\renewcommand{\epsilon}{\varepsilon}
\renewcommand{\div}{\mathop{\rm div}\nolimits}

\newcommand{\<}{\langle}
\renewcommand{\>}{\rangle}
\renewcommand{\Cap}{{\rm cap}}

\title[Metastability from the large deviations point of view]{
$\Gamma$-expansion of the measure-current large deviations rate
functional of non-reversible finite-state Markov chains}

\author{S. Kim}
\address{Seonwoo Kim
  \hfill\break\indent Korea Institute for Advanced Study \hfill\break\indent 85 Hoegi-ro, \hfill\break\indent
Dongdaemun-gu, Seoul 02455, South Korea.} 
\email{seonwookim@kias.re.kr}

\author{C. Landim}
\address{Claudio Landim
  \hfill\break\indent IMPA \hfill\break\indent Estrada Dona Castorina
  110, \hfill\break\indent
J. Botanico, 22460 Rio de Janeiro, Brazil\hfill\break\indent
  {\normalfont and} \hfill\break\indent CNRS UMR 6085, Universit\'e de
  Rouen, \hfill\break\indent Avenue de l'Universit\'e, BP.12,
  Technop\^ole du Madril\-let, \hfill\break\indent
F76801 Saint-\'Etienne-du-Rouvray, France.} 
\email{landim@impa.br}

\begin{document}

\begin{abstract}
Consider a sequence of continuous-time Markov chains
$(X^{(n)}_t:t\ge 0)$ evolving on a fixed finite state space $V$.  Let
$I_n$ be the measure-current large deviations rate functional for
$X^{(n)}_t$, as $t\to\infty$.  Under a hypothesis on the jump rates,
we prove that $I_n$ can be written as
$I_n = \mb I^{(0)} \,+\, \sum_{1\le p\le \mf q} (1/\theta^{(p)}_n) \,
\mb I^{(p)}$ for some rate functionals $\mb I^{(p)}$. The weights
$\theta^{(p)}_n$ correspond to the time-scales at which the sequence
of Markov chains $X^{(n)}_t$ evolves among the metastable wells, and
the rate functionals $\mb I^{(p)}$ characterise the asymptotic
Markovian dynamics among these wells. This expansion provides
therefore an alternative description of the metastable behavior of a
sequence of Markovian dynamics. Together with the
results in \cite{bgl-24,l-gamma}, this work finishes the project of
characterising the hierarchical metastable behavior of finite-state
Markov chains by means of the $\Gamma$-expansion of large deviations
rate functionals.
In addition, we present optimal conditions
under which the measure (Donsker-Varadhan) or the measure-current large deviations rate functional
determines the original dynamics, and calculate the first and second derivatives of
the measure large deviations rate functional, thereby generalising the results
for i.i.d. random variables.
\end{abstract}

\noindent
\keywords{Large deviations, empirical current, $\Gamma$-convergence, metastability}

\subjclass[2010]
{Primary 60K35, 60F10, 82C22}

\maketitle
\thispagestyle{empty}

\section{Introduction}
\label{sec1}

Let $V$ be a finite set. Denote by $\color{blue} (X_t : t\ge 0)$
a $V$-valued, irreducible continuous-time Markov chain,
whose jump rates are represented by $\color{blue} R(x,y)$.  The
generator reads as
\begin{equation*}
{ \color{blue} (\ms L f)(x)}
\,=\, \sum_{y\in V} R(x,y)\, \{\,
f(y)\,-\, f(x)\,\}\;, \quad f\colon V \to \bb R\;.
\end{equation*}
Let $\color{blue} \pi$ be the unique stationary state.  The so-called
Matrix tree Theorem \cite[Lemma 6.3.1]{FW98} provides a representation
of the measure $\pi$ in terms of arborescences of the set $V$.

Denote by $\color{blue} \ms P(V)$ the space of probability measures on
$V$ endowed with the weak topology, and by $L_t$ the empirical
measure of the chain $X_t$ defined as:
\begin{equation}
\label{48}
{ \color{blue} L_t}
\,:=\, \frac{1}{t} \int_0^t \delta_{X_s}\, {\rm d}\,s \;,
\end{equation}
where $\color{blue} \delta_x$, $x\in V$, represents the Dirac measure
concentrated at $x$. Thus, $L_t$ is a random element of
$\ms P(V)$ and $L_t (V_0)$, $V_0 \subset V$, stands for
the average amount of time the process $X_t$ stays at $V_0$ in the
time interval $[0,t]$.

As the Markov chain $X_t$ is irreducible, by the ergodic theorem, for
any starting point $x\in V$, as $t\to\infty$, the empirical measure
$L_t$ converges in probability to the stationary state $\pi$.

Denote by $E$ the set of directed edges:
\begin{equation}
\label{03}
{\color{blue} E} \,:=\,  \big \{ \, (x,y) \in V\times V: y\neq x\,,\,
R(x,y)>0 \, \big \}\;.
\end{equation}
and by $(Q_t : t\ge 0)$ the empirical flow
defined by
\begin{equation*}
{\color{blue}Q_t (x,y) } \,:=\,
\frac{1}{t}\, \sum_{0<s\le t} \mb 1 \, \{ X_{s-} = x\,,\,
X_{s} = y\}\;.
\end{equation*}
In words, $t\, Q_t (x,y)$ counts the number of times the
process $X$ jumped from $x$ to $y$ in the time interval $[0,t]$.
By \cite[Section 2.1]{bfg-ihp}, for any starting point $x\in V$, as
$t\to\infty$, $Q_t (x,y)$ converges in probability to
$\pi(x)\, R(x,y)$.

A function $J\colon E \to \bb R_+$ is called a \emph{flow}. The set of
flows defined on $E$ is represented by $\color{blue} \mf F_E$.  We
sometimes refer to a flow in $\mf F_E$ as an $E$-flow.
The divergence of a flow $J$ at a vertex $x\in V$, denoted by
$(\div J)(x)$, is given by
\begin{equation*}
(\div J)(x) \,=\, \sum_{y: (x,y)\in E} J(x,y) \;-\;
\sum_{y: (y,x)\in E} J(y,x)\;.
\end{equation*}
A flow is said to be \emph{divergence-free} if $(\div J)(x) =0$ for
all $x\in V$. Let $\color{blue} \mf F_E^{\div}$ be the set of
divergence-free $E$-flows.

Denote by $\color{blue} D(\bb R_+, W)$, $W$ a finite set, the space of
right-continuous functions $\mf f \colon \bb R_+ \to W$ with
left-limits endowed with the Skorohod topology and its associated
Borel $\sigma$-algebra. Let $\color{blue} \mb P_{\! x}$, $x\in V$, be
the probability measure on the path space $D(\bb R_+, V)$ induced by
the Markov chain $X_t$ starting from $x$. Expectation with
respect to $\mb P_{\! x}$ is represented by $\color{blue} \mb E_x$.

Let $\Phi\colon \bb R_+ \times \bb R_+ \to [0,+\infty]$ be the
function defined by
\begin{equation}
\label{Phi}
{\color{blue} \Phi (q,p)} \,:=\,
\begin{cases}
\, p  & \text{if $q=0$}\;,
\\
\, \displaystyle{ q \log \frac qp - (q-p)}
& \text{if $q\in (0,+\infty)$ and $p\in (0,+\infty)$}\;,
\\
\, +\infty & \text{if $q\in (0,+\infty)$ and $p=0$}\;. 
\end{cases}
\end{equation}
For $p>0$, $\Phi( \cdot, p)$ is a nonnegative convex function which
vanishes only at $q=p$. Actually, $\Phi( \cdot, p)$ is the large
deviations rate functional of a Poisson process with parameter $p$.

Denote by $J_{\mu, R} \in \mf F_{E}$, $\mu\in \ms P(V)$, the flow
defined by
\begin{equation}
\label{19}
{\color{blue} J_{\mu, R} (x,y)} \,:=\, \mu(x)\, R(x,y)\;, \quad
(x,y)\,\in\,E\;.
\end{equation}
Let $\Upsilon_{E,R} \colon \ms P(V)\times \mf F_{E} \to [0,+\infty]$
be the functional defined by
\begin{equation}
\label{02c}
{\color{blue} \Upsilon_{E,R} (\mu,J)} \,:=\,
\sum_{(x,y)\in E} \Phi \big(\, J(x,y) \,,\, J_{\mu, R}(x,y) \,\big) \;.
\end{equation}
Let $I\colon \ms P(V)\times \mf F_E \to [0,+\infty]$ be the
measure-current large deviations rate functional defined by
\begin{equation}
\label{02}
{\color{blue} I (\mu,J) } \,:=\,
\begin{cases}
\Upsilon_{E,R} (\mu,J) 
& J\,\in\,\mf F^{\div}_E \;,
\\
\; +\infty  & \text{otherwise}\;.
\end{cases}
\end{equation}
Bertini, Faggionato and Gabrielli \cite{bfg-ihp} proved a large
deviations principle for the pair $(L_t, Q_t)$. Mind that Condition
2.2 in \cite{bfg-ihp} is trivially satisfied in the case where $V$ is
finite. 

\begin{theorem}
\label{t01}
For every closed set $C$ of $\ms P(V) \times E$, and every
open set $G$ of $\ms P(V) \times E$,
\begin{gather*}
\limsup_{t\to\infty} \, \sup_{x\in V} \, \frac{1}{t} \, \log \, \mb P_{\! x}\big[\, (L_t, Q_t) \, \in \,
C\,\big] \,\le\, -\, \inf_{(\mu, J) \in C} \, I(\mu, J)\;,
\\
\liminf_{t\to\infty} \, \inf_{x\in V} \, \frac{1}{t} \, \log \, \mb P_{\! x} \big[\, (L_t, Q_t) \, \in \,
G\,\big] \,\ge\, -\, \inf_{(\mu, J) \in G} \, I(\mu, J)\;.
\end{gather*}
Moreover, $I$ is a convex rate function with compact level sets.
\end{theorem}

By \cite[Theorem 1.6]{bfg}, the projection of the rate functional
$I(\mu, J)$ on the first coordinate yields the Donsker-Varadhan large
deviations functional \cite{dv75} for the empirical measure:
\begin{equation}
\label{60}
\inf_{J\in \mf F^{\div}_E} \, I(\mu, J) \,=\,
\sup_{u>0} \, - \, \int \frac{\ms Lu}{u}\, {\rm d}\,\mu
\,=:\, {\color{blue} \ms I (\mu)} \;.
\end{equation}

\subsection*{$\Gamma$-convergence}

Suppose now that $\color{blue} (X^{(n)}_t : t\ge 0)$ is a sequence of
$V$-valued, irreducible continuous-time Markov chains, whose jump
rates are represented by $\color{blue} R_n(x,y)$. We add a subscript
$n$ to the previous notation to refer to the Markov chain
$X^{(n)}$. In particular, the generator, the stationary state, and the
rate functionals are represented by $\color{blue} \ms L_n$,
$\color{blue} \pi_n$, $\color{blue} I_n$, and $\color{blue} \ms I_n$,
respectively.  We assume, however, that the set of directed edges $E$
introduced in \eqref{03} does not depend on $n$: for all $n\ge 1$,
\begin{equation}
\label{16}
R_n(x,y) \,>\, 0 \quad \text{if and only if}\quad
(x,y) \,\in \, E \; .
\end{equation}
In particular, for all $n\ge 1$,
\begin{equation}
\label{62}
{\color{blue} I_n (\mu,J) }\,:=\,
\begin{cases}
\Upsilon_{E,R_n} (\mu,J) 
& J\,\in\,\mf F^{\div}_E \;,
\\
\; +\infty  & \text{otherwise}\;.
\end{cases}
\end{equation}
We assume furthermore that the jump rates satisfy condition \eqref{mh}. 

In this article, we investigate the $\Gamma$-convergence of the
measure-current large deviations rate functional $I_n$.  Referring to
\cite{Br} for an overview, we recall the definition of
$\Gamma$-convergence.  Fix a Polish space $\mc X$ and a sequence
$(U_n : n\in\bb N)$ of functionals on $\mc X$,
$U_n\colon \mc X \to [0,+\infty]$. The sequence $U_n$
\emph{$\Gamma$-converges} to the functional
$U\colon \mc X\to [0,+\infty]$, i.e.
$U_n \stackrel{\Gamma}{\longrightarrow} U$, if and only if the two
following conditions are met:
\begin{itemize}
\item [(i)]\emph{$\Gamma$-liminf.} The functional $U$ is a
$\Gamma$-liminf for the sequence $U_n$: For each $x\in\mc X$ and each
sequence $x_n\to x$, we have that $\liminf_n U_n(x_n) \ge U(x)$.

\item [(ii)]\emph{$\Gamma$-limsup.} The functional $U$ is a
$\Gamma$-limsup for the sequence $U_n$: For each $x\in\mc X$ there
exists a sequence $x_n\to x$ such that $\limsup_n U_n(x_n) \le U(x)$.
\end{itemize}

The main result of the article provides a $\Gamma$-expansion of the
rate functional $I_n$.  It states that there exist sequences
$(\theta^{(p)}_n \colon n\ge 1)$, $1\le p\le \mf q$, and large deviations
rate functionals
$\mb I^{(p)}\colon \ms P(V)\times \mf F_E \to [0,+\infty]$,
$0\le p\le \mf q$ such that $\theta^{(1)}_n\to \infty$,
$\theta^{(p)}_n/\theta^{(p+1)}_n \to 0$, $1\le p <\mf q$, and 
$\theta^{(p)}_n I_n$ $\Gamma$-converges to $\mb I^{(p)}$. We summarize
this result by writing a $\Gamma$-expansion for $I_n$:
\begin{equation}
\label{14}
I_n \,=\, \mb I^{(0)} \;+\; \sum_{p=1}^{\mf q}
\frac{1}{\theta^{(p)}_n} \, \mb I^{(p)}\;.
\end{equation}

The rate function $I_n$ thus encodes all the characteristics of the
metastable behavior of the sequence of Markov chains $X^{(n)}_t$. The
factors $\theta^{(p)}_n$ provide the time-scales at which it is
observed and the rate functions $\mb I^{(p)}$ correspond to the
generators describing the synthetic evolution.

This article completes a program started in \cite{bgl-24, l-gamma},
which set out to characterise the metastable behavior of finite state
continuous time Markov chains through the $\Gamma$-convergence of the
large deviations rate functionals.

Next result is a simple consequence of the large deviations principle
stated in Theorem \ref{t01} and the $\Gamma$-convergence informally
described in the previous paragraphs.  (cf. Corollary 4.3 in
\cite{mar}).

\begin{corollary}\label{cor2}
Fix $0\le p\le \mf q$ and set $\theta^{(0)}_n=1$.
For every closed subset $C$ of $\ms P(V) \times E$ and every
open subset $G$ of $\ms P(V) \times E$,
\begin{gather*}
\limsup_{n\to \infty} \,  \limsup_{t\to\infty} \,
\frac{\theta^{(p)}_n}{t} \, \sup_{x\in V} \, \log \, \mb P^n_{\! x} \big[\, (L_t, Q_t) \,\in \,
C\,\big] \,\le\, -\, \inf_{(\mu, J) \in C} \, \mb I^{(p)} (\mu, J)\;,
\\
\liminf_{n\to \infty} \,  \liminf_{t\to\infty} \,
\frac{\theta^{(p)}_n}{t} \, \inf_{x\in V} \, \log \, \mb P^n_{\!x} \big[\, (L_t, Q_t) \,\in \,
G\,\big] \,\ge\, -\, \inf_{(\mu, J) \in G} \, \mb I^{(p)} (\mu, J)\;.
\end{gather*}
\end{corollary}

The remainder of the article is organised as follows. In Section \ref{sec2}, we state the main result of this work, the $\Gamma$-convergence $\theta_n^{(p)} I_n \to {\bf I}^{(p)}$ for all $0 \le p \le \mf q$. In Section \ref{sec3} we prove the $p=0$ case; in Section \ref{sec4} we prove the remaining $1 \le p \le \mf q$ cases. In Section \ref{sec5}, we analyze the conditions under which the two large deviations rate functionals (cf. \eqref{02} and \eqref{60}) determine the Markov chain, including explicit counterexamples which demonstrate that the conditions are optimal.

\section{Notation and results}
\label{sec2}

In this section we state the main result of the article. This requires
some notation.  Denote by $\color{blue} \lambda_n(x)$, $x\in V$, the
holding rates of the Markov chain $X^{(n)}_t$ and by
$\color{blue} p_n(x,y)$, $x,y\in V$, the jump probabilities, so
that $R_n(x,y) = \lambda_n(x) \, p_n(x,y)$.  The generator reads
therefore as
\begin{equation*}
{ \color{blue} (\ms L_n f)(x)}
\,=\, \sum_{y\in V} R_n(x,y)\, \{
f(y)\,-\, f(x)\}\;, \quad f\colon V \to \bb R\;.
\end{equation*}

Let $E'$ be a proper subset of $E$. Denote by $\mf F_{E'}$ the set of
flows in $\mf F_E$ such that $J(x,y)=0$ for all
$(x,y) \in E\setminus E'$:
\begin{equation}
\label{15}
\mf F_{E'} \,:=\, \big\{ \, J\in \mf F_E:
J(x,y)=0 \quad \text{for all} \quad (x,y) \in E\setminus E' \,
\big\}\;.
\end{equation}
As before, denote by $\mf F^{\div}_{E'}$ the elements of
$\mf F_{E'}$ which are divergence-free.

\subsection*{The $\Gamma$-convergence}

Assume that $\lim_n R_n(x,y)$ exists for all $(x,y)\in E$ and denote
by $\bb R_{0} (x,y) \in [0,\infty)$ its limit:
\begin{equation}
\label{01}
{\color{blue} \bb R_{0} (x,y)}
\,:=\, \lim_{n \to \infty} \, R_n(x,y)\;, \quad (x,y) \,\in\, E \;. 
\end{equation}
Let $\bb E_0$ be the set of edges whose asymptotic rate is positive:
\begin{equation*}
{\color{blue} \bb E_0} \, := \, \{\, (x,y)\in E : \bb R_{0}
(x,y) \, > \, 0\,\} \; ,
\end{equation*}
and assume that $\bb E_0 \not = \varnothing$.  The jump
rates $\bb R_0 (x,y)$ induce a continuous-time Markov chain on $V$,
denoted by $\color{blue} (\bb X_t:t\ge 0)$, which, of course, may be
reducible. Denote by $\color{blue} \bb L^{(0)}$ its generator.

Denote by $\color{blue} \ms V_{1}, \dots, \ms V_{\mf n}$,
$\mf n\ge 1$, the closed irreducible classes of $\bb X_t$, and let
\begin{equation}
\label{05}
{\color{blue} S} \,:=\, \{1, \dots, \mf n\}\;, \quad 
{\color{blue} \ms V}  \,:=\, \bigcup_{j\in S} \ms V_{j} \;, \quad
{\color{blue} \Delta } \,:=\,  V  \, \setminus \, \ms V \;. 
\end{equation}
The set $\Delta$ may be empty and some of the sets $\ms V_j$ may be
singletons.

Let $\mb I^{(0)} \colon \ms P(V)\times \mf F_{E} \to [0,+\infty]$ be the
functional given by
\begin{equation}
\label{09}
\mb I^{(0)} (\mu,J) \,:=\,
\begin{cases}
\displaystyle{\Upsilon_{\bb E_0, \bb R_0} (\mu,J)} 
& J \,\in\, \mf F^{\div}_{\bb E_0} \;,
\\
\; +\infty  & \text{otherwise}\;,
\end{cases}
\end{equation}
which is the measure-current large deviations rate functional of
$\bb X_t$.

\begin{proposition}
\label{p01}
The functional $I_n$ $\Gamma$-converges to $\mb I^{(0)}$.
\end{proposition}

We prove Proposition \ref{p01} in Section \ref{sec3}.

\subsection*{The main assumption}

To examine the $\Gamma$-convergence of the measure-current large
deviations rate functionals at longer time-scales, we introduce a
natural hypothesis on the jump rates proposed in \cite{bl4} and
adopted in \cite{lx16, fk17, bgl-24, l-gamma}.

For two sequences of positive real numbers $(\alpha_n : n\ge 1)$,
$(\beta_n : n\ge 1)$, $\color{blue} \alpha_n \prec \beta_n$ or
$\color{blue} \beta_n \succ \alpha_n$ means that
$\lim_{n\to\infty} \alpha_n/\beta_n = 0$. Similarly,
$\color{blue} \alpha_n \preceq \beta_n$ or
$\color{blue} \beta_n \succeq \alpha_n$ indicates that either
$\alpha_n \prec \beta_n$ or $\alpha_n/\beta_n$ converges to a positive
real number $a\in (0,\infty)$.

Two sequences of positive real numbers $(\alpha_n : n\ge 1)$,
$(\beta_n : n\ge 1)$ are said to be \emph{comparable} if
$\alpha_n \prec \beta_n$, $\beta_n \prec \alpha_n$ or
$\alpha_n / \beta_n \to a \in (0,\infty)$. This condition excludes the
possibility that
$\liminf_n \alpha_n /\beta_n \neq \limsup_n \alpha_n /\beta_n$.

A set of sequences
$(\alpha^{\mf u}_n: n\ge 1)$, $\mf u \in \mf R$, of positive real
numbers, indexed by some finite set $\mf R$, is said to be comparable
if for all $\mf u,\mf v \in\mf R$ the sequences
$(\alpha^{\mf u}_n : n \ge 1)$, $(\alpha^{\mf v}_n : n \ge 1)$ are
comparable.

Recall that we denote by $E$ the set of directed edges (independent of
$n$) with positive jump rates.
Let $\bb Z_+ = \{0, 1, 2, \dots \}$, and $\color{blue} \Sigma_m$,
$m\ge 1$, be the set of functions $k\colon E \to \bb Z_+$ such that
$\sum_{(x,y)\in E} k(x,y) = m$.  We assume, hereafter, that for every
$m\ge 1$ the set of sequences
\begin{equation}
\label{mh}
\Big(\, \prod_{(x,y)\in E} R_n(x,y)^{k(x,y)} : n \ge 1 \,\Big)
\;,\quad k\,\in\,\Sigma_m \;,
\end{equation}
is comparable. In Remark \ref{rm1}, we comment on this assumption. 

\subsection*{Tree decomposition}

If the Markov chain $\bb X_t$ has only one closed irreducible class,
%(sometimes called below ergodic class)
the $\Gamma$-expansion of
$I_n$ has only one term, $\mb I^{(0)}$. Indeed, in this case, by Lemma
\ref{l04}, $\mb I^{(0)}(\mu, J)=0$ implies that $\mu$ is the stationary
state of the Markov chain $\bb X_t$ and $J = J_{\mu, \bb R_0}$.
In particular, as $I_n$ converges to $\mb I^{(0)}$, for any sequence
$\beta_n\to\infty$, $\beta_n \,I_n (\mu_n, J_n) \to \infty$ for any
sequence $(\mu_n, J_n)$ converging to $(\mu, J) \neq (\pi, J_{\pi, \bb
R_0})$ if $\pi$ represents the stationary state of the Markov chain
$\bb X_t$. 

Assume therefore that there are more than one closed irreducible
class, in other words, that the constant $\mf n$ introduced in
\eqref{05} is larger than or equal to $2$: $\mf n\ge 2$.  Under this
assumption, \eqref{16}, and \eqref{mh}, \cite{bl7, lx16} constructed a
rooted tree which describes the behaviour of the Markov chain
$X^{(n)}_t$ at all different time-scales. We recall the construction
below.

The tree satisfies the following conditions:

\begin{itemize}
\item[(a)] Each vertex of the tree represents a subset of $V$;
\item[(b)] Each generation forms a partition of $V$;
\item[(c)] The children of each vertex form a partition of the parent;
\item[(d)] The generation $p+1$ is strictly coarser than the
generation $p$.
\end{itemize}

The tree is constructed by induction starting from the leaves to the
root.  It corresponds to a deterministic coalescence process. Denote
by $\color{blue} \mf q$ the number of steps in the recursive
construction of the tree. At each level $1\le p\le \mf q$, the
procedure generates a partition
$\{ \ms V^{(p)}_1, \dots, \ms V^{(p)}_{\mf n_p}, \Delta_p\}$, a
time-scale $\theta^{(p)}_n$ and a $\{1, \dots, \mf n_p\}$-valued
continuous-time Markov chain $\bb X^{(p)}_t$ which describes the
evolution of the chain $X^{(n)}_{t \theta^{(p)}_n}$ among the subsets
$\ms V^{(p)}_1, \dots, \ms V^{(p)}_{\mf n_p}$, called hereafter
\emph{wells}.

The leaves are the sets $\ms V_1, \dots, \ms V_{\mf n}, \Delta$
introduced in \eqref{05}. We proceed by induction. Let
$\color{blue} S_1 = S$, $\color{blue} \mf n_1 = \mf n$,
$\color{blue} \ms V^{(1)}_j = \ms V_j$, $j\in S_1$,
$\color{blue} \Delta_1 = \Delta$, and assume that the recursion has
produced the sets
$\ms V^{(p)}_1, \dots, \ms V^{(p)}_{\mf n_p}, \Delta_p$ for some
$p\ge 1$, which forms a partition of $V$.

Denote by $H_{\ms A}$, $H^+_{\ms A}$, ${\ms A}\subset V$, the hitting
and return time of ${\ms A}$:
\begin{equation} 
\label{201}
{\color{blue} H_{\ms A} } \,:=\,
\inf\, \big \{\,t>0 : X^{(n)}_t \in {\ms A}\, \big\}\;,
\quad
{\color{blue} H^+_{\ms A}} \,:=\,
\inf\, \big \{\,t>\tau_1 : X^{(n)}_t \in {\ms A}\, \big\}\; ,  
\end{equation}
where $\tau_1$ represents the time of the first jump of the chain
$X^{(n)}_t$:
$\color{blue} \tau_1 = \inf\,\{t>0 : X^{(n)}_t \not = X^{(n)}_0\}$.

For two non-empty, disjoint subsets $\ms A$, $\ms B$ of $V$, denote by
$\Cap_n(\ms A, \ms B)$ the capacity between $\ms A$ and $\ms B$:
\begin{equation}
\label{202}
{\color{blue} \Cap_n(\ms A, \ms B)}
\,:=\, \sum_{x\in \ms A} \pi_n(x)\, \lambda_n(x) 
\, \mb P^n_{\! x} \big[\, H_{\ms B} < H^+_{\ms A}\, \big]\;.
\end{equation}
Set $\color{blue} S_p = \{1, \dots, \mf n_p\}$, and let
$\theta^{(p)}_n$ be defined by
\begin{equation}
\label{26b}
{\color{blue} \frac 1{\theta^{(p)}_n}} \,:=\,
\sum_{i\in S_p}  \frac{\Cap_n (\ms V^{(p)}_i, \breve{\ms V}^{(p)}_i)}
{\pi_n(\ms V^{(p)}_i)}  \;, \quad
\text{where}\;\;
{\color{blue} \breve{\ms V}^{(p)}_i}  \,:=\, \bigcup_{j\in S_p \setminus\{ i\}} \ms
V^{(p)}_j\;. 
\end{equation}
By \cite[Assertion 8.B]{lx16},
\begin{equation}
\label{51}
\theta^{(p-1)}_n \;\prec\;  \theta^{(p)}_n \;.
\end{equation}

The ratio
$\pi_n(\ms V^{(p)}_i) /\Cap_n (\ms V^{(p)}_i, \breve{\ms V}^{(p)}_i)$
represents the time it takes for the chain $X^{(n)}_t$, starting from
a point in $\ms V^{(p)}_i$ to reach the set $\breve{\ms
V}^{(p)}_i$. Therefore, $\theta^{(p)}_n$ corresponds to the smallest
time needed to observe such a jump.

%Let $\color{blue} \ms V^{(p)} = \cup_{j\in S_p} \ms V^{(p)}_j$.

Let $\Psi_p\colon V\to S_p \cup \{0\}$ be the projection which sends
the points in $\ms V^{(p)}_j$ to $j$ and the elements of $\Delta_p$ to
$0$:
\begin{equation*}
{\color{blue} \Psi_p} \,:=\, \sum_{k\in S_p} k\; \chi_{_{\ms V^{(p)}_k}}\;.
\end{equation*}
In this formula and below, $\color{blue} \chi_{_{\ms A}}$ stands for
the indicator function of the set $\ms A$.  Next theorem follows from
the main result in \cite{lx16} and \cite{llm18}.

\begin{theorem}
\label{t1}
Assume that conditions \eqref{16}, \eqref{mh} are in force. Then, for
each $j\in S_p$, $x\in \ms V^{(p)}_j$, under the measure
$\mb P^n_{\! x}$, the finite-dimensional distributions of the sequence
of $(S_p \cup\{0\})$-valued processes
$\Psi_p (X^{(n)}_{t\theta^{(p)}_n})$ converge to the
finite-dimensional distributions of a $S_p$-valued Markov chain,
represented by $\bb X^{(p)}_t$.
\end{theorem}

The process $\bb X^{(p)}_t$ describes therefore how the chain
$X^{(n)}_{t}$ evolves among the wells $\ms V^{(p)}_j$ in the
time-scale $\theta^{(p)}_n$.  Note that the Markov chain
$\bb X^{(p)}_t$ takes value in $S_p$, while the process
$\Psi_p (X^{(n)}_{t\theta^{(p)}_n})$ may also be equal to $0$.

Denote by $\color{blue} r^{(p)}(j,k)$ the jump rates of the
$S_p$-valued continuous-time Markov chain
$\color{blue} (\bb X^{(p)}_t : t\ge 0)$.  By \cite[Theorem 2.7]{lx16},
there exist $j$, $k\in S_p$ such that $r^{(p)} (j,k) >0$. Actually, by
the proof of this result,
\begin{equation}
\label{36}
\text{$\displaystyle \sum_{k\not = j} r^{(p)} (j,k) >0$
for all $j\in S_p$ such that}\quad 
\lim_{n\to \infty} \theta^{(p)}_n \,
\frac{\Cap_n (\ms V^{(p)}_j, \breve{\ms V}^{(p)}_j)}
{\pi_n(\ms V^{(p)}_j)} >0 \;.
\end{equation}

Denote by
$\color{blue} \mf R^{(p)}_1, \dots, \mf R^{(p)}_{\mf n_{p+1}}$ the
recurrent classes of the $S_p$-valued chain $\bb X^{(p)}_t$, and by
$\color{blue} \mf T_p$ the transient states. Let
${\color{blue} \mf R^{(p)}} = \cup_j \mf R^{(p)}_j$, and observe that
$\{\mf R^{(p)}_1, \dots, \mf R^{(p)}_{\mf n_{p+1}}, \mf T_p\}$ forms a
partition of the set $S_p$. This partition of $S_p$ induces a new
partition of the set $V$. Let
\begin{equation*}
{\color{blue}  \ms V^{(p+1)}_m}
\,:=\, \bigcup_{j\in \mf R^{(p)}_m} \ms V^{(p)}_j\;, \quad
{\color{blue}  \ms T^{(p+1)}}
\,:=\, \bigcup_{j\in \mf T_p} \ms V^{(p)}_j\;, 
\quad m\in  {\color{blue} S_{p+1} \,:=\, \{1, \dots, \mf n_{p+1}\}} \;,
\end{equation*}
so that $V  \,=\,  \Delta_{p+1}   \, \cup \, \ms V^{(p+1)}$, where
\begin{equation}
\label{05b}
{\color{blue} \ms V^{(p+1)}}
\,=\, \bigcup_{m\in S_{p+1}} \ms V^{(p+1)}_{m}\;,
\quad
{\color{blue}  \Delta_{p+1}}
\,:=\, \Delta_p \,\cup\, \ms T^{(p+1)} \;.
\end{equation}

The subsets
$\ms V^{(p+1)}_1, \dots, \ms V^{(p+1)}_{\mf n_{p+1}}, \Delta_{p+1}$ of
$V$ are the result of the recursive procedure. We claim that
conditions (a)--(d) hold at step $p+1$ if they are fulfilled up to
step $p$ in the induction argument.

The sets
$\ms V^{(p+1)}_{1}, \dots, \ms V^{(p+1)}_{\mf n_{p+1}}$,
$\Delta_{p+1}$ constitute a partition of $V$ because the sets
$\mf R^{(p)}_1, \dots, \mf R^{(p)}_{\mf n_{p+1}}$, $\mf T_p$ form a
partition of $S_p$, and the sets
$\ms V^{(p)}_{1}, \dots, \ms V^{(p)}_{\mf n_{p}}$, $\Delta_{p}$ one of
$V$. Conditions (a)--(c) are therefore satisfied.

To show that the partition obtained at step $p+1$ is strictly coarser
than $\{\ms V^{(p)}_{1}, \dots,$
$\ms V^{(p)}_{\mf n_{p}}, \Delta_{p}\}$, observe that, by \eqref{36},
$r^{(p)}(j,k)>0$ for some $k\neq j \in S_p$. Hence, either $j$ is a
transient state for the process $\bb X^{(p)}_t$ or the closed
recurrent class which contains $j$ also contains $k$. In the first
case $\Delta_p \subsetneq \Delta_{p+1}$, and in the second one there
exists $m\in S_{p+1}$ such that
$ \ms V^{(p)}_{j} \cup \ms V^{(p)}_{k} \subset \ms V^{(p+1)}_{m}$.
Therefore, the new partition
$\{\ms V^{(p+1)}_{1}, \dots, \ms V^{(p+1)}_{\mf n_{p+1}},
\Delta_{p+1}\}$ of $V$ satisfies the condition (d).

The construction terminates when the $S_p$-valued Markov chain
$\bb X^{(p)}_t$ has only one recurrent class so that $\mf
n_{p+1}=1$. In this situation, the partition at step $p+1$ is
$\ms V^{(p+1)}_{1}$, $\Delta_{p+1}$.

This completes the construction of the rooted tree.  Recall that we
denote by $\color{blue} \mf q$ the number of steps of the scheme.  As
claimed at the beginning of the procedure, for each
$1\le p\le \mf q$, we generated a time-scale $\theta^{(p)}_n$, a
partition
$\ms P_p = \{\ms V^{(p)}_{1}, \dots, \ms V^{(p)}_{\mf n_{p}},
\Delta_{p}\}$, where
$\ms P_1 = \{\ms V_{1}, \dots, \ms V_{\mf n}, \Delta \}$,
$\ms P_{\mf q+1} = \{\ms V^{(\mf q+1)}_{1}, \Delta_{\mf q+1} \}$, and
a $S_p$-valued continuous-time Markov chain $\bb X^{(p)}_t$.

The partitions $\ms P_1, \dots, \ms P_{\mf q+1}$ form a rooted tree
whose root ($0$-th generation) is $V$, first generation is
$\{\ms V^{(\mf q+1)}_{1}, \Delta_{\mf q+1} \}$ and last
($(\mf q+1)$-th) generation is
$\{\ms V_{1}, \dots, \ms V_{\mf n}, \Delta \}$.  Note that the set
$\ms V^{(p+1)}$ corresponds to the set of recurrent points for the
chain $\bb X^{(p)}_t$. In contrast, the points in $\Delta_{p+1}$ are
either transient for this chain or negligible in the sense that the
chain $X^{(n)}_t$ remains a negligible amount of time on the set
$\Delta_p$ in the time-scale $\theta^{(p)}_n$ (cf. \cite{bl4, lx16}).

\subsection*{A set of measures}

We construct in this subsection a set of probability measures
$\pi^{(p)}_j$, $1\le p\le \mf q +1$, $j\in S_p$, on $V$ which
describe the evolution of the chain $X^{(n)}_t$ and such that
\begin{equation}
\label{o-52}
\text{ the support of $\pi^{(p)}_j$ is the set $\ms V^{(p)}_j$}\; .
\end{equation}

We proceed by induction. Let $\color{blue} \pi^{(1)}_j$, $j\in S_{1}$, be the
stationary states of the Markov chain $\bb X_t$ restricted to the
closed irreducible classes $\ms V^{(1)}_j = \ms V_j$. Clearly, condition
\eqref{o-52} is fulfilled.

Fix $1\le p\le \mf q$, and assume that the probability measures
$\pi^{(p)}_j$, $j\in S_p$, have been defined and satisfy condition
\eqref{o-52}. Denote by $\color{blue} M^{(p)}_m(\cdot)$,
$m\in S_{p+1}$, the stationary state of the Markov chain
$\bb X^{(p)}_t$ restricted to the closed irreducible class
$\mf R^{(p)}_m$. The measure $ M^{(p)}_m$ is understood as a measure
on $S_p=\{1, \dots, \mf n_p\}$ which vanishes on the complement of
$\mf R^{(p)}_m$.  Let $\pi^{(p+1)}_m$ be the probability measure on
$V$ given by
\begin{equation}
\label{o-80}
{\color{blue} \pi^{(p+1)}_m (x)}
\,:=\, \sum_{j\in \mf R^{(p)}_m} M^{(p)}_m(j)\, \pi^{(p)}_j (x)\;, \quad
x\in V\;.
\end{equation}

Clearly, condition \eqref{o-52} holds, and the measure
$\pi^{(p+1)}_m$, $1\le p\le \mf q$, $m\in S_{p+1}$, is a convex
combination of the measures $\pi^{(p)}_j$, $j\in \mf R^{(p)}_m$.
Moreover, by \cite[Theorem 3.1 and Proposition 3.2]{bgl-24}, for all
$z\in \ms V^{(p)}_j$,
\begin{equation}
\label{o-58}
\lim_{n\to\infty} \frac{\pi_n(z)}{\pi_n(\ms V^{(p)}_j)}
\,=\, \pi^{(p)}_j(z) \,\in\, (0,1] \;, \quad
\lim_{n\to \infty} \pi_n(\Delta_{\mf q+1}) \,=\, 0
\;.
\end{equation}

By \eqref{o-80}, the measures $\pi^{(p)}_j$, $2\le p\le \mf q+1$,
$j\in S_p$, are convex combinations of the measures $\pi^{(1)}_k$,
$k\in S_1$. By \eqref{o-58}, for all $x\,\in\, \ms V^{(\mf q+1)}$,
$\lim_{n\to\infty} \pi_n(x)$ exists and belongs to $(0,1]$. By
\eqref{o-58}, and since by (1.c) $\Delta_p \subset \Delta_{p+1}$ for
$1\le p\le \mf q$, $\lim_{n\to \infty} \pi_n(\Delta_{p}) \,=\, 0$ for
all $p$.

\subsection*{The $\Gamma$-expansion}

We are now in a position to state the main result of this article.
Let $\color{blue} \bb L^{(p)}$, $1\le p\le \mf q$, be the generator of
the $S_p$-valued Markov chain $\bb X^{(p)}_t$.  Denote by
$\color{blue} \ms P(S_p)$, $1\le p\le \mf q$, the set of probability
measures on $S_p$. Let
$\bb I^{(p)} \colon \ms P (S_p) \to [0,+\infty]$ be the
Donsker-Varadhan large deviations rate functional of $\bb X^{(p)}_t$
given by
\begin{equation}
\label{40}
{\color{blue} \bb I^{(p)} (\omega) } \,:=\,
\sup_{\mb h} \,-\,  \sum_{j\in S_p} \omega_j \,
e^{- \mb h (j) } \, (\bb L^{(p)} e^{ \mb h})(j)  \;,
\end{equation}
where the supremum is carried over all functions
$\mb h:S_p \to \bb R$.  Denote by
$\mb I^{(p)} \colon \ms P(V) \times \mf F_E \to [0,+\infty]$ the
functional given by
\begin{equation}
\label{o-83b}
{\color{blue} \mb I^{(p)} (\mu, J) } \,:=\,
\left\{
\begin{aligned}
& \bb I^{(p)} (\omega)   \quad \text{if}\;\;
\mu = \sum_{j\in S_p} \omega_j \, \pi^{(p)}_j \;\; \text{for}\;\;
\omega \in \ms P (S_p)\;\; \text{and}\;\; J = J_{\mu, \bb R_0} \;,  \\
& +\infty \quad\text{otherwise}\;.
\end{aligned}
\right.
\end{equation}

The main result of the article reads as follows.

\begin{theorem}
\label{mt1}
For each $1\le p\le \mf q$, the functional $\theta^{(p)}_nI_n$
$\Gamma$-converges to $\mb I^{(p)}$.
\end{theorem}

We complete this section with some comments. %on condition \eqref{mh}.

\begin{remark}
\label{rm2}
It follows from the previous result that it is too costly to modify
the current. More precisely, fix a pair $(\mu, J)$ of measure and
current.  The cost (that is the value of the large deviations rate
functional) of the pair $(\mu, J)$ for a current $J$ different from
the one induced by the measure $\mu$ (that is $J_{\mu, \bb R_0}$) is
finite only on the initial scale. On all the other ones it is infinite.
For this reason, in the time-scales $\theta^{(p)}_n$, $p \ge 1$, in the
proof of the upper bound we may restrict the analysis to the optimal
current $J_n^*$ associated to the recovery sequence of measures $\nu_n$
of $\mu$ constructed in \cite{l-gamma}. Since the measure-current
large deviations rate functional computed at the optimal current is
equal to the Donsker-Varadhan large deviations rate functional of the
measure (see equation \eqref{60}), the $\Gamma$-convergence of the
pairs measure-current is reduced to the $\Gamma$-convergence of the
measures.
\end{remark}

\begin{remark}
\label{rm3}
It follows from \eqref{02c} that
\begin{equation}
\label{61}
\theta\, \Upsilon_{E,R} (\mu,J) \,=\,
\Upsilon_{E, \theta  R} (\mu, \theta J) \;.
\end{equation}
Since multiplying the jump rates by a constant corresponds to speeding
up the dynamics by the same amount, by considering the $\Gamma$-limit of
$\theta_n\, I_n$, we are actually examining the asymptotic behaviour
of the large deviations rate functional of the process on the longer
time-scale $\theta_n$.
\end{remark}

In view of equation \eqref{61}, one might be tempted to investigate
the limit of $\Upsilon_{E, \theta R} (\mu, J)$, where we only
accelerate the process and not the currents. However, by speeding-up
the dynamics, we also amplify the typical flows.

Indeed, since $\Phi(q,p) = p \, \phi (q/p)$ for $p$,
$q\in (0,\infty)$, where $\phi(x) = x\log x +1 -x$, for bonds $(x,y)$
such that $\theta_n \, \mu_n (x) \, R_n(x,y) \to \infty$, for
$\Phi(J_n(x,y), \theta_n \, \mu_n (x) \, R_n(x,y))$ to converge to a finite
real number, $J_n(x,y)$ must diverge (more precisely, the ratio
$J_n(x,y)/ ( \theta_n \, \mu_n (x) \, R_n(x,y) )$ must converge to $1$).  Thus,
by investigating the asymptotic behavior of
$\Upsilon_{E, \theta R} (\mu, J)$, one would need to consider flows of
the order $\theta_n \, \mu_n (x) \, R_n(x,y)$, which is exactly what is
considered in the asymptotic analysis of
$\Upsilon_{E, \theta R} (\mu, \theta J)$.

\begin{remark}
\label{rm5}
The metastable time-scales are the time-scales at which one observes a
modification in the structure of the empirical measures. There might
exist intermediate scales, between metastable time-scales, at which
one observes a modification in the structure of the flows. These
intermediate time-scales are not captured by the analysis carried out
in this article.

To illustrate the above assertion, consider the reversible Markov
chain $X^{(n)}_t$ taking values in $E=\{-3, \dots, 3\}$ with only
nearest-neighbour jumps whose rates are given by
$R_n(-2,-1) = R_n(-1,0) = R_n(2,1) = R_n(1,0) =1/n$, all the other
ones being equal to $1$.  The stationary state, denoted by $\pi_n$, is
the measure $\pi_n(-3) =\pi_n(-2) =\pi_n(2) =\pi_n(3) = a_n$,
$\pi_n(-1) = \pi_n(-1) = a_n/n$, $\pi_n(0) = a_n/n^2$, where $a_n$ is
a normalising constant.

In this example, as jumps are nearest neighbour, the divergence free
flows are symmetric. Besides the initial time-scale, $n^2$ is the only
other metastabe time scale. In this time-scale jumps between the wells
$\{-3, -2\}$ and $\{2,3\}$ are observed.

We turn to the flows.  In times of order $1$, flows between sites
$-3$, $-2$ and between sites $3$, $2$ are observed. As the jump rate
from $-2$ to $-1$ is $1/n$ no flow is observed between $-2$ and $-1$,
and, for the same reasons, between $-1$ and $0$, between $0$ and $1$, or between $1$
and $2$.  This analysis is confirmed by Proposition \ref{p01}, which,
applied to this example, states that the measure-current large
deviations rate functional $I_n$, introduced in \eqref{62},
$\Gamma$-converges to the rate functional $\mb I^{(0)}$ associated to
the chain, denoted by $\mb X^{(0)}_t$, obtained from $X^{(n)}_t$ by
setting to $0$ all jump rates equal to $1/n$. The only flows with
finite cost are precisely those between sites $-3$, $-2$ and between
$3$, $2$. All the other ones have infinite cost.

At time of order $n$, one observes a flow of order $1$ between $-2$
and $-1$, and a flow of order $n$ between $-3$ and $-2$. That is, in
any time interval $[0,t]$, $t>0$, with positive probability, there is
at least one jump from $-2$ to $-1$ followed by one from $-1$ to
$-2$.  In this example, as claimed in the remark, there is an
intermediate time-scale, $n$ as argued above, at which the structure
of the flows changes, while the empirical measures do not.
\end{remark}

\begin{remark}
\label{rm1}
The hypothesis \eqref{mh} on the jump rates is taken from \cite{bl4} and
\cite{lx16}. It is a natural condition in the investigation of the
asymptotic behavior of sequences of Markov chains.  Indeed, a first
reasonable hypothesis to impose consists in assuming that the jump
rates converge, as stated in \eqref{01}.

In longer time-scales, the process is expected to remain long times in
a set of wells and to perform very short excursions among the points
which separate the wells. It is therefore also natural to require the
jump rates of the trace process to converge. (We refer to
\cite[Section 6]{bl2}, \cite{lrev} for the definition of the trace
process.)

Fix $V_0\subset V$, and denote by $R^{V_0}_n$ the jump rates of the
trace in $V\setminus V_0$ of the Markov chain $X^{(n)}_t$.  According
to the displayed equation after Corollary 6.2 in \cite{bl2}, $R^{V_0}_n$
can be expressed as a sum of products of terms of the form
\begin{equation}
\label{17}
\frac{R_n(x_0, x_1) \cdots R_n(x_{m-1}, x_m)}
{\sum_{b=1}^q R_n(y^b_0, y^b_1) \cdots R_n(y^b_{m-1}, y^b_m)}
\end{equation}
for some $m\ge 1$.  Here $x_i$, $y^b_j\in V$. Mind that the number
of terms in each product, $m$, is always the same.

Condition \eqref{mh} is precisely the one needed to guarantee that
such expressions have a limit (which might be $+\infty$) and do not
oscillate.

Finally, as observed in \cite{bl4} (see Remark 2.2 in \cite{bgl-24}),
assumption \eqref{mh} is fulfilled by all statistical mechanics models
which evolve on a fixed finite state space and whose metastable
behaviour has been derived. This includes the Ising model \cite{ns91,
ns92, bc96, bm02}, the Potts model with or without a small external
field \cite{nz19, ks22, ks24}, the Blume-Capel model \cite{co96, ll16}, and
conservative Kawasaki dynamics \cite{bhn06, GHNO09, HNT12, bl15b}.
\end{remark}

\section{The initial time-scale}
\label{sec3}

In this section, we investigate the initial, $0$-th, time-scale. In detail, we prove
Proposition \ref{p01} and identify the set of zeros of the functional $\mb I^{(0)}$.

A flow $J\colon E \to \bb R_+$ is called a \emph{cycle} if there
exists a set of distinct edges $(x_0, x_1), \dots, (x_{n-1}, x_n)$ in
$E$ and a constant $a\not = 0$ such that $x_n=x_0$ and
\begin{equation*}
J(x,y) \,=\,
\begin{cases}
a & \text{if $(x,y) = (x_j, x_{j+1})$ for some $0\le j<n$}\;,
\\
0 & \text{otherwise}\;.
\end{cases}
\end{equation*}
Clearly, every cycle is divergence-free and every divergence-free flow
can be expressed as a finite sum of cycles. To prove the latter argument,
it suffices to observe the following. Given any $J \in \mf F_E^{\div}$,
take $(x,y)\in E$ such that
\begin{equation}
\label{cyc}
J(x,y) \,  = \, \min \big \{ \, J(z,w) : (z,w) \in E , \;\; J(z,w)>0 \, \big \} \; ,
\end{equation}
and find a set of distinct edges $(x_0,x_1) , \dots , (x_{n-1},x_n)$ with
$(x_0,x_1) = (x,y)$ and $x_n=x_0$ such that $J(x_j,x_{j+1})>0$ for each $0 \le j < n$,
which is possible since $J$ is divergence-free. Define a cycle $J_0 \colon E \to \bb R_+$
which has value $J(x,y)$ along the $n$ edges described above. Then, by the minimality in \eqref{cyc},
$J-J_0 \colon E \to \bb R_+$ is again a divergence-free flow which has strictly less number of edges with
a positive value. Iterating this procedure, which ends in finite steps, we obtain the
desired decomposition as a sum of cycles.

%Fix a probability measure $\mu\in\ms P(V)$, and let $\mf F^\mu_E$ be
%the set of flows in $E$ that vanish at the directed edges whose
%endpoints do not belong to the support of $\mu$:
%\begin{equation*}
%{\color{blue} \mf F^\mu_E } \,:=\, \big\{\, J \in \mf F_E :
%\text{ $J(x,y) = 0$ for all $(x,y)\in E$ such
%that $\{x,y\} \not\subset {\rm supp}\; \mu$} \,\big\}\;.
%\end{equation*}
%With this notation, we may rewrite the BFG rate functional $I$
%introduced in \eqref{02} as
%\begin{equation}
%\label{02b}
%I (\mu,J) \,=\,
%\begin{cases}
%\displaystyle{\Upsilon_{E,R} (\mu,J)}
%& \text{if $J\,\in\,\mf F^{\div}_E \cap \mf F^\mu_E$\; ,}
%\\
%\; +\infty  & \text{otherwise}\,.
%\end{cases}
%\end{equation}

\begin{proof}[Proof of Proposition \ref{p01}]
We first consider the $\Gamma$-limsup. Fix a pair $(\mu,J)$ in
$\ms P(V)\times \mf F_{E}$. We may assume that
$\mb I^{(0)} (\mu,J) < \infty$, otherwise there is nothing to
prove. Under this restriction, $J$ belongs to
$\mf F^{\div}_{\bb E_0}$, thus to $\mf F^{\div}_{E}$.

Let $(\mu_n,J_n)$ be the sequence constantly equal to $(\mu,J)$.  Since
$J$ belongs to $\mf F^{\div}_{E}$, by
\eqref{02c} and \eqref{02},
\begin{equation*}
I_n (\mu,J) \,=\, \Upsilon_{E, R_n} (\mu,J)
\,=\, \sum_{(x,y)\in E} \Phi \big(\, J(x,y) \,,\, J_{\mu,R_n}(x,y) \,\big) \;.
\end{equation*}
It remains to show that
\begin{equation}
\label{04b}
\lim_{n\to\infty} \Phi \big(\, J(x,y) \,,\, J_{\mu,R_n}(x,y) \,\big)
\,=\, \Phi \big(\, J(x,y) \,,\, J_{\mu,\bb R_0}(x,y) \,\big ) \quad
\text{for all $(x,y)\in E$}\; .
\end{equation}

Suppose that $J(x,y)=0$. By \eqref{Phi}, $\Phi(0, \cdot)$ is continuous. Thus, by
\eqref{01} and the definition \eqref{19} of the flow
$J_{\mu,R_n}(x,y) $, \eqref{04b} holds for edges $(x,y)\in E$ such that
$J(x,y)=0$. On the other hand, if $J(x,y)>0$, since $\mb I^{(0)} (\mu,J) < \infty$, by
\eqref{Phi}, $\mu(x)\, \bb R_0(x,y)  = J_{\mu,\bb R_0}  (x,y) >
0$. Thus, by \eqref{01}, \eqref{04b} also holds in this case. This
completes the proof of the $\Gamma$-limsup. Note that we proved that
\begin{equation*}
\lim_{n\to \infty} I_n (\mu,J) \,=\, \mb I^{(0)} (\mu,J)\;.
\end{equation*}

We turn to the $\Gamma$-liminf. Fix a pair $(\mu,J)$ in
$\ms P(V)\times \mf F_{E}$ and a sequence $(\mu_n,J_n)$ in
$\ms P(V)\times \mf F_{E}$ converging to $(\mu,J)$.  If $J$ is not
divergence-free, for $n$ sufficiently large $J_n$ is not
divergence-free as well and for those $n$'s
$I_n (\mu_n,J_n) \, = \, \mb I^{(0)} (\mu,J) \, = \, \infty$.

Assume that $J$ is divergence-free and that $J(x,y)>0$ for some
$(x,y) \not \in \bb E_0$. In this case $\mb I^{(0)}(\mu,J) = \infty$
(because $J_{\mu,\bb R_0} (x,y) = \mu(x) \, \bb R_0(x,y) = 0$ so that,
in view of \eqref{Phi},
$\Phi (\, J(x,y) \, , \, J_{\mu,\bb R_0} (x,y) \, ) = +\infty$). On the other
hand, since $\Phi$ is positive, $ J_n(x,y) \to J(x,y)>0$ and
$J_{\mu_n,R_n}(x,y) \le R_{n}(x,y) \to 0$,
\begin{equation*}
\liminf_{n\to\infty} I_n (\mu_n,J_n) \;\ge \;
\liminf_{n\to\infty} \Phi \big(\, J_n(x,y) \,,\, J_{\mu_n,R_n}(x,y)
\,\big) \,=\, \infty \,=\, \mb I^{(0)}(\mu,J)  \;.
\end{equation*}

The previous arguments show that we may restrict our attention to
divergence-free flows that vanish on edges which do not belong to
$\bb E_0$, that is, to flows in $\mf F^{\div}_{\bb E_0}$. Assume
that $J$ belongs to this set. Since $\Phi$ is positive,
\begin{equation*}
\liminf_{n\to\infty} I_n (\mu_n,J_n) \;\ge \;
\liminf_{n\to\infty}
\sum_{(x,y)\in \bb E_0} \Phi \big(\, J_n(x,y) \,,\,
J_{\mu_n,R_n}(x,y) \,\big) \;. 
\end{equation*}
Fix $(x,y)\in \bb E_0$. We consider three cases. If $\mu(x)>0$,
$\Phi (\, J_n(x,y) \,,\, J_{\mu_n,R_n}(x,y) \,)$ converges to
$\Phi (\, J(x,y) \,,\, J_{\mu,\bb R_0}(x,y) \,)$. If $\mu(x)=0$ and
$J(x,y)=0$, then, as $\Phi$ is positive,
\begin{equation*}
\liminf_{n\to\infty} \Phi \big(\, J_n(x,y) \,,\, J_{\mu_n,R_n}(x,y) \, \big )
\,\ge\, 0 \, = \, \Phi \big(\, J(x,y) \,,\, J_{\mu,\bb R_0} (x,y) \, \big)\;.
\end{equation*}
Finally, if $\mu(x)=0$ and $J(x,y)>0$, then
$\liminf_{n\to\infty} \Phi (\, J_n(x,y) \,,\, J_{\mu_n,R_n}(x,y) \,)
\,=\, \infty \,=\, \Phi (\, J(x,y) \,,\, J_{\mu,\bb R_0} (x,y) \,)$.
This completes the proof of the proposition.
\end{proof}

\subsection*{The functional $\mb I^{(0)}$}

The set of divergence-free flows in $\mf F_{\bb E_0}^{\div}$ has a simple
structure.  Denote by $\bb E_{0,j}$, $1\le j\le \mf n$, the set of
directed edges in $\bb E_0$ whose both endpoints belong to $\ms V_j$:
\begin{equation*}
\bb E_{0,j} \,:=\, \big \{\, (x,y) \in \bb E_0 : x\,,\,y \,\in\, \ms
V_j\,\big\}\;. 
\end{equation*}
%The set $\bb E_0 \setminus ( \cup_j \bb E_{0,j})$ might be non empty
%because there might be edges $(x,y)$ in $\bb E_0$ whose tail $x$ is a
%$\bb R_0$-transient state.
We say that $y$ is equivalent to $x$, $\color{blue} y \sim x$ if $y=x$ or if there exist sequences $x=x_0,\dots,x_\ell=y$ and $y=y_0,\dots,y_m=x$ such that $\bb{R}_0(x_i,x_{i+1})>0$, $\bb{R}_0(y_j,y_{j+1})>0$ for all $0\le i < \ell$, $0\le j < m$. Denote by $\color{blue} \ms C_1,\dots,\ms C_{\mf m}$ the equivalent classes which are not one of the sets $\ms V_j$, $1\le j \le \mf n$. Then, denote by $\color{blue} \bb E_{0,k}^{\rm{Tr}}$, $1\le k\le \mf m$, the set of directed edges in $\bb E_0$ whose both endpoints belong to $\ms C_k$:
\begin{equation*}
\bb E_{0,k}^{\rm Tr} \,:=\, \big \{\, (x,y) \in \bb E_0 : x\,,\,y \,\in\, \ms
C_k\,\big\}\;. 
\end{equation*}

We claim that any $\bb E_0$-divergence-free flow is a sum of
$\bb E_{0,j}$- or $\bb E_{0,k}^{\rm Tr}$-divergence-free flows (that is, flows whose edges belong
to $\bb E_{0,j}$ or $\bb E_{0,k}^{\rm Tr}$, respectively).
To see this, suppose the contrary that $J(a_1,a_2) > 0$ for some $(a_1,a_2) \in \bb E_0$ which does not belong to any $\bb E_{0,j}$ or $\bb E_{0,k}^{\rm Tr}$. Then, there exist two distinct collections $\ms A , \ms A_1 \in \{ \ms V_1 , \dots , \ms V_{\mf n} , \ms C_1 , \dots , \ms C_{\mf m} \}$ such that $a \in \ms A$ and $a_1 \in \ms A_1$. Note that $J(a,a_1)>0$ implies $\bb R_0(a,a_1) > 0$ (cf. $J \in \mf F_{\bb E_0}^{\rm div}$), thus $\ms A$ should be a transient class.

Recall that $J \in \mf F_{\bb E_0}^{\div}$ implies
\begin{equation}
\label{df}
\sum_{y:(x,y) \in \bb E_0} J(x,y) \, = \, \sum_{y:(y,x) \in \bb E_0} J(y,x) \quad \text{for all $x \in V$} \; .
\end{equation}
Adding \eqref{df} up for all $x\in \ms A_1$, all edges whose both endpoints belonging to $\ms A_1$ cancel out with each other and we obtain
\begin{equation*}
\sum_{x \in \ms A_1} \sum_{y \notin \ms A_1 : (x,y) \in \bb E_0} J(x,y) \, = \, \sum_{x \in \ms A_1} \sum_{y \notin \ms A_1 : (y,x) \in \bb E_0} J(y,x) \; .
\end{equation*}
Since $J(a,a_1) > 0$ and $a \notin \ms A_1$, the right-hand side is positive, thus the left-hand side is also positive. This guarantees the existence of another collection $\ms A_2 \in \{ \ms V_1,\dots,\ms V_{\mf n},\ms C_1,\dots,\ms C_{\mf m} \}$ such that $J$ is positive along some edge from $\ms A_1$ to $\ms A_2$. Furthermore, $\ms A_2$ is also distinct from the previous collections $\ms A$ or $\ms A_1$ since $\ms A \ne \ms A_1$.

We may apply the same logic to each $\ms A_{\ell-1}$, $\ell\ge2$, recursively, to obtain a new collection $\ms A_\ell$ which is distinct from all previous $\ell$ collections. Since $V$ is finite and $\ms A$ is not an irreducible class, there exists $\ell \ge 1$ such that $\ms A_\ell$ is an irreducible class. Then, adding \eqref{df} up for all $x \in \ms A_\ell$, we obtain
\begin{equation*}
\sum_{x \in \ms A_\ell} \sum_{y \notin \ms A_\ell : (x,y) \in \bb E_0} J(x,y) \, = \, \sum_{x \in \ms A_\ell} \sum_{y \notin \ms A_\ell : (y,x) \in \bb E_0} J(y,x) \; .
\end{equation*}
By the inductive procedure, the right-hand side is positive, but the left-hand side is zero since $\ms A_\ell$ is irreducible, yielding a contradiction. Thus, we conclude that $J(x,y)>0$ only for those $(x,y)$ belonging to $\bb E_{0,j}$ or $\bb E_{0,k}^{\rm Tr}$.

%Each $\bb E_{0,j}$-divergence-free flow can be
%further decomposed as a linear combination of $\bb
%E_{0,j}$-cycles, and the same for $\bb E_{0,k}^{\rm Tr}$-divergence free flows. Moreover, if
%\begin{equation*}
%J \,=\, \sum_{j=1}^{\mf n} J_j + \sum_{k=1}^{\mf m} J_k^{\rm Tr} \;, 
%\end{equation*}
%where $J_j \in \mf F^{\div}_{\bb E_{0,j}}$ and $J_k^{\rm Tr} \in \mf F^{\div}_{\bb E_{0,k}^{\rm Tr}}$, then for every $\mu\in \ms P(V)$, we rewrite
%$\mb I^{(0)}(\mu,J)$ as
%\begin{equation}
%\label{04}
%\begin{aligned}
%\sum_{j=1}^{\mf n} \sum_{(x,y)\in \bb E_{0,j}} &
%\Phi \big(\, J_j(x,y) \,,\, J^\mu_{\bb R_0}(x,y) \,\big) +\; \sum_{k=1}^{\mf m} \sum_{(x,y)\in \bb E_{0,k}^{\rm Tr}}
%\Phi \big(\, J_k^{\rm Tr}(x,y) \,,\, J^\mu_{\bb R_0}(x,y) \,\big)
%\\
%& \;+\; \sum_{(x,y)\in \bb E_0 \setminus (\bigcup_{j} \bb E_{0,j} \cup \bigcup_k \bb E_{0,k}^{\rm Tr})}
%\mu(x)\, \bb R_0(x,y) \;.
%\end{aligned}
%\end{equation}
%In particular,
%\begin{equation*}
%\bb I_{0} (\mu,0) \,=\, \sum_{(x,y)\in \bb E_0} \mu(x)\, \bb
%R_0(x,y)\;. 
%\end{equation*}

Now, we are ready to characterise the zeros of the functional
$\mb I^{(0)}$.  Recall from the definition after \eqref{o-52} that
$\pi^{(1)}_j$, $1\le j\le \mf n$ denotes the invariant probability
measure on $\ms V_j$ for the $\bb R_0$-chain.

\begin{lemma}
\label{l04}
Fix a probability measure $\mu$ in
$\ms P(V)$ and a flow $J$ in $\mf F^{\div}_{\bb E_0}$. Then,
$\mb I^{(0)} (\mu,J) =0$ if and only if there exist weights $\omega_j$,
$1\le j\le \mf n$, such that $\omega_j\ge 0$, $\sum_j \omega_j =1$,
$\mu = \sum_j \omega_j \, \pi^{(1)}_j$, and $J = J_{\mu,\bb R_0}$.
\end{lemma}

\begin{proof}
First, assume that $\mb I^{(0)} (\mu,J) = 0$ and recall formula
\eqref{09}. Since $\Phi$ is positive,
$\Phi (\, J(x,y) \,,\, J_{\mu,\bb R_0}(x,y) \,) =0$ for all
$(x,y)\in \bb E_0$. As $\Phi (q,p) =0$ if, and only if, $q=p$,
$J(x,y) = J_{\mu,\bb R_0}(x,y)$ for all $(x,y)\in \bb E_0$.  As
$J$ is a divergence-free $\bb E_{0,j}$-flow (cf. $\bb E_{0,k}^{\rm Tr}$-flow) when restricted to each
$\ms V_j$ (cf. $\ms C_k$), so is $J_{\mu,\bb R_0}$. This implies that $\mu$ is
stationary for each $\ms V_j$- or $\ms C_k$-valued Markov chain with jump rates
$\bb R_0$.

Next, fix a transient collection $\ms C_k$.
There exists at least one $x\in \ms C_k$ such that $\bb R_0(x,y)>0$
for some $y \notin \ms C_k$. As $\ms C_k$ is a transient class,
$x\in \ms C_k$, $y \notin \ms C_k$, and
$J\in \mf F^{\div}_{\bb E_0}$, $J(x,y)$ has to vanish.  Since
$J(x,y)=0$ and $\bb R_0(x,y)>0$ we have $\mu(x)=0$. Since $\mu$ is stationary inside $\ms C_k$,
$\mu=0$ on the whole $\ms C_k$.

In summary, $\mu$ is a convex
combination of only the ergodic measures on $\ms V$ (cf. \eqref{05}), that are,
$\pi^{(1)}_j$, $1\le j\le \mf n$, which proves the only if part.

Finally, to show the if part, suppose that $\mu = \sum_j \omega_j \, \pi^{(1)}_j$ for some
probability measure $\omega$ on $S_1=\{1,\dots,\mf n\}$ and that $J=J_{\mu,\bb R_0}$. Then, $\mu$ is stationary with respect to the $\bb R_0$-chain, thus $J \in \mf F_{\bb E_0}^{\rm div}$.
By \eqref{09} and \eqref{02c},
\begin{equation*}
\mb I^{(0)} (\mu,J) \,  =  \, \sum_{(x,y) \in \bb E_0} \Phi \big( \, J(x,y) \, , \, J_{\mu,\bb R_0}(x,y) \, \big) = 0 \; ,
\end{equation*}
where the second equality follows since, by \eqref{Phi}, $\Phi(q,p)=0$ if and only if $q=p$. This concludes the proof of Lemma \ref{l04}.
\end{proof}

\section{Longer time-scales}
\label{sec4}

We start with a lemma which states that the functionals $\mb I^{(p)}$ (cf. \eqref{o-83b}),
$1 \le p \le \mf q$, form a hierarchical structure of zeros.
Recall from \eqref{40} that $\bb I^{(p)}$, $1\le p\le \mf q$, denotes the
Donsker-Varadhan large deviations rate functional of $\bb X_t^{(p)}$.

\begin{lemma}\label{hier}
For each $1 \le p \le \mf q$, ${\bf I}^{(p)}(\mu,J)<\infty$ if and only if ${\bf I}^{(p-1)}(\mu,J)=0$.
\end{lemma}
\begin{proof}
Suppose first that ${\bf I}^{(p)}(\mu,J)<\infty$, so that
by \eqref{o-83b}, $\mu=\sum_{j=1}^{\mf n_{p}}\omega_{j} \, \pi_{j}^{(p)}$
and $J=J_{\mu,\bb R_{0}}$. Then, $\bb I^{(p)}(\omega) = \mb I^{(p)}(\mu,J_{\mu,\bb R_{0}})<\infty$,
thus by \eqref{o-80} and \cite[Lemma 5.1 and Eq. (5.1)]{l-gamma}, we have ${\bf I}^{(p-1)}(\mu,J_{\mu,\bb R_{0}})=0$
for $p\ge2$. If $p=1$, we also have ${\bf I}^{(0)}(\mu,J_{\mu,\bb R_{0}})=0$
by Lemma \ref{l04}.

Next, suppose that ${\bf I}^{(p-1)}(\mu,J)=0$.
If $p=1$, then by Lemma \ref{l04} we have $\mu=\sum_{j=1}^{\mf n}\omega_{j} \,\pi_{j}^{(1)}$
and $J=J_{\mu,\bb R_{0}}$, so that ${\bf I}^{(1)}(\mu,J)=\bb I^{(1)}(\omega)<\infty$
by \eqref{o-83b}. If $p\ge2$, then
by \eqref{o-83b}, $\mu=\sum_{j=1}^{\mf n_{p-1}}\omega_{j} \, \pi_{j}^{(p-1)}$,
$J=J_{\mu,\bb R_{0}}$, and moreover $\bb I^{(p-1)}(\omega)=0$.
Via \cite[Lemma 5.1]{l-gamma} and again \eqref{o-83b}, we obtain that ${\bf I}^{(p)}(\mu,J_{\mu,\bb R_0})<\infty$, 
finishing the verification of Lemma \ref{hier}.
\end{proof}

Now, we provide the proof of Theorem \ref{mt1}.

\begin{proof}[Proof of Theorem \ref{mt1}]
We proceed by induction. Suppose that the statement holds for $p-1$,
$p\ge 1$, where the $0$-th step holds by Proposition \ref{p01}.

Recall that $\ms I_{n}\colon \ms P(V)\to[0,+\infty]$ denotes the
Donsker-Varadhan large deviations rate functional of the original
process $X_t^{(n)}$. Denote by $\color{blue} \ms I^{(p)}$ the Donsker-Varadhan
large deviations rate functional of $\bb X_t^{(p)}$.
Recall from \cite[Theorem 2.5]{l-gamma} that
$\theta_{n}^{(p)} \ms I_{n}$ $\Gamma$-converges to $\ms I^{(p)}$.

First, we consider the $\Gamma$-limsup. Fix
$(\mu,J)\in\ms P(V)\times\mf F_{E}$.  We may assume that
$\mu=\sum_{1\le j\le \mf n_{p}}\omega_{j} \, \pi_{j}^{(p)}$, for some
probability measure $\omega$ in $S_p$, and $J=J_{\mu,\bb
R_{0}}$. Otherwise, by \eqref{o-83b},
$\mb I^{(p)}(\mu, J) = + \infty$, and there is nothing to prove.
Furthermore, by Lemmata B.4 and B.5 in \cite{l-gamma}, it is enough to
prove the theorem for such measures $\mu$ with $\omega_j>0$ for all $j\in S_p$.

Therefore, fix from now on
$\mu=\sum_{1\le j\le \mf n_{p}}\omega_{j} \, \pi_{j}^{(p)}$ for some
$\omega \in \ms P(S_p)$ such that $\omega_j>0$ and
$J=J_{\mu,\bb R_{0}}$. By the proof of the $\Gamma$-convergence
$\theta_n^{(p)} \ms I_n \to \ms I^{(p)}$ presented in \cite[Section 5]{l-gamma},
there exists a sequence of probability measures $(\nu_{n}:n\ge1)$ on
$V$ such that $\nu_n(x)>0$ for all $x\in V$, that converges to
$\mu$ and satisfies
\begin{equation*}
\limsup_{n\to\infty} \, \theta_{n}^{(p)}
\ms I_{n}(\nu_{n}) \, \le \, \ms I^{(p)}(\mu) \, = \, {\bf I}^{(p)} (\mu, J_{\mu,\bb R_0}) \; ,
\end{equation*}
where the equality holds by \eqref{o-83b} and \cite[Eq. (2.18)]{l-gamma}.
By \cite[Theorem 1.6]{bfg}, there exists a unique flow
$J_n^* \in \mf F_E$ such that
\begin{equation}
\label{opt}
\ms I_n (\nu_n) \, = \, I_n (\nu_n , J_n^*) \; .
\end{equation}
Thus, by \cite[Theorem 2.5]{l-gamma} and \eqref{o-83b}, 
\begin{equation*}
\limsup_{n\to\infty} \, \theta_{n}^{(p)} I_{n}(\nu_{n},J_{n}^*)
\,=\, \limsup_{n\to\infty} \, \theta_{n}^{(p)} \ms I_n (\nu_n)
\, \le \, {\bf I}^{(p)}(\mu,J_{\mu,\bb R_{0}}) \; .
\end{equation*}
To conclude the $\Gamma$-limsup part, it remains to prove that
$\lim_{n\to\infty}J_{n}^*=J_{\mu,\bb R_{0}}$.

Since $\nu_n(x)>0$ for all $x\in V$, by \cite[Eq. (2.9)]{bfg}, there
exists a function $g_n \colon V \to \bb R$ such that
\begin{equation}\label{gn}
\log \frac{J_n^*(x,y)}{ J_{\nu_n,R_n}(x,y) } \,
= \, g_n (y) \, - \, g_n (x) \quad \text{for each $(x,y) \in E$} \; .
\end{equation}
By the definition of $I_n$, \eqref{opt}, and \cite[Eq. (2.17)]{bfg},
\begin{equation}
\label{63}
\ms I_n (\nu_n) \, = \, \sum_{(x,y) \in E} \Upsilon_{E,R_n} ( \nu_n , J_n^* ) \, = \,  \sum_{(x,y) \in E} \nu_n(x) \, R_n(x,y) \, \big( \, 1 \, - \, e^{g_n(y) - g_n(x)} \, \big) \; .
\end{equation}
Thus, by \cite[Lemma A.3]{l-gamma}, 
\begin{equation}
\label{pos}
\ms I_n (\nu_n) \, = \, \sum_{(x,y) \in E}
\nu_n (x) \, R_n(x,y) \, \big( \, (g_n(y)\,-\,g_n(x))
\, e^{g_n(y)-g_n(x)} \, - \, e^{g_n(y) - g_n(x)} \, + \, 1 \,
\big)\;. 
\end{equation}

Since $\nu_n$ converges to $\mu$ and $R_n$ to $\bb R_0$,
$\nu_n (x) \, R_n(x,y)$ converges to
$\mu(x) \bb R_0(x,y) = J_{\mu, \bb R_0} (x,y)$. Denote by
$\widehat{\bb E}_0 \subset \bb E_0$ the set of edges $(x,y)$ such that
$J_{\mu, \bb R_0} (x,y)>0$.

As $\nu_n \to \mu$ and $\ms I^{(0)} (\mu) =0$ by Lemma \ref{l04},
$\lim_n \ms I_n (\nu_n) = \ms I^{(0)} (\mu)= 0$. Since all terms
on the right-hand side of \eqref{pos} are nonnegative, for all
$(x,y) \in \widehat{\bb E}_0$,
$\lim_{n\to \infty} [g_n(y)\,-\,g_n(x)] =0$. Thus, by \eqref{gn}, for
such edges, $\lim_{n\to \infty} J^*_n(x,y) = \lim_{n\to \infty}
J_{\nu_n, R_n}(x,y) =  J_{\mu, \bb R_0} (x,y)$. 

It remains to consider the edges $(x,y) \in E \setminus \widehat{\bb
E}_0$. By \eqref{63} and the previous paragraph,
\begin{equation*}
0\,=\, \lim_{n\to\infty} \ms I_n (\nu_n) \, = \,
\lim_{n\to\infty} \sum_{(x,y) \in E\setminus \widehat{\bb E}_0 }
\nu_n(x) \, R_n(x,y) \, \big( \, 1 \, - \, e^{g_n(y) - g_n(x)} \, \big) \; .
\end{equation*}
For edges in $E\setminus \widehat{\bb E}_0$,
$\lim_{n\to\infty} \nu_n (x) \, R_n(x,y) =0$. Thus, by \eqref{gn},
\begin{equation*}
\begin{aligned}
0\,=\, 
 -\, \lim_{n\to\infty} \sum_{(x,y) \in E\setminus \widehat{\bb E}_0 }
\nu_n(x) \, R_n(x,y) \, e^{g_n(y) \, - \, g_n(x)} 
\,=\, -\, \lim_{n\to\infty} \sum_{(x,y) \in E\setminus \widehat{\bb E}_0 }
J^*_n(x,y)  \; .
\end{aligned}
\end{equation*}
Since the flows are non-negative, $J^*_n(x,y) \to 0 = J_{\mu, \bb
R_0}(x,y)$ for all $(x,y) \in E\setminus \widehat{\bb E}_0$. This completes
the proof of the $\Gamma$-limsup.

Next, we consider the $\Gamma$-liminf. Fix $(\mu,J)\in\ms P(V)\times\mf F_{E}$
and an arbitrary sequence $(\mu_{n},J_{n})\in\ms P(V)\times\mf F_{E}$,
$n\ge1$, that converges to $(\mu,J)$. We have
\begin{equation}\label{eq3}
\liminf_{n\to\infty} \, \theta_{n}^{(p)} \, \ms I_{n}(\mu_{n}) \, \ge \, \ms I^{(p)}(\mu) \; .
\end{equation}
If $\mb I^{(p)}(\mu,J)<\infty$, such that $J=J_{\mu,\bb R_0}$ by \eqref{o-83b}, then we easily obtain that
\[
\liminf_{n\to\infty} \, \theta_{n}^{(p)} \, I_{n}(\mu_{n},J_{n})\,\ge\,\liminf_{n\to\infty}\,\theta_{n}^{(p)} \, \ms I_{n}(\mu_{n})\,\ge\,\ms I^{(p)}(\mu)\,=\,{\bf I}^{(p)}(\mu,J_{\mu,\bb R_0}) \; ,
\]
where the first inequality holds by \eqref{60}, the second inequality
holds by \eqref{eq3}, and the equality holds by \eqref{o-83b} and \cite[Eq. (2.18)]{l-gamma}.
Finally, suppose that $\mb I^{(p)}(\mu,J)=\infty$. Then by Lemma
\ref{hier}, ${\bf I}^{(p-1)}(\mu,J)>0$. By the induction hypothesis at
level $p-1$,
\begin{equation*}
\liminf_{n\to\infty} \, \theta_{n}^{(p-1)} \, I_{n}(\mu_{n},J_{n})\,\ge\,{\bf I}^{(p-1)}(\mu,J)\,>\,0 \; .
\end{equation*}
Since $\theta_{n}^{(p-1)} \prec \theta_{n}^{(p)}$ by \eqref{51}, we conclude that
\begin{equation*}
\liminf_{n\to\infty} \,\theta_{n}^{(p)} \, I_{n}(\mu_{n},J_{n}) \, = \, \infty\, = \, {\bf I}^{(p)}(\mu,J) \; ,
\end{equation*}
finishing the proof of Theorem \ref{mt1}.
\end{proof}

\section{The large deviations rate functional}
\label{sec5}

This section is divided into three parts.
In the first part, we prove that the Donsker-Varadhan, or DV,
large deviations rate functional characterises the dynamics if the
Markov chain is \emph{reversible} (Corollary \ref{ap-l02}); in contrast, it does not
characterise the dynamics if it is non-reversible (Example \ref{ex01}).
In the second part, in the measure-current formalism, we may ease the reversibility condition;
namely, the measure-current, or BFG for Bertini, Faggionato and Gabrielli, large
deviations rate functional characterises the dynamics as long as \emph{all elements are recurrent} (Corollary \ref{ap-l9}), and does not otherwise (Example \ref{ex02}). Indeed, every state in a reversible Markov chain is recurrent, but not vice versa.
In the third part, we calculate the first and second functional derivates of the DV rate functional which extend the well-known analogues for the large deviations rate functional of i.i.d. random variables to the setting of finite-state Markov chains. 

The term ``characterise'' above should be interpreted as follows. Fix a graph structure $(V,E)$ and suppose that there are two Markov chains $(X_t : t \ge 0)$ and $(X_t' : t \ge 0)$ defined thereon. Denote by $\ms I, I$ and $\ms I', I'$ the DV/BFG rate functionals of $X_t$ and $X_t'$, respectively. Then, we aim to find optimal conditions on $(V,E)$ such that $\ms I = \ms I'$, or $I = I'$, implies $(X_t : t \ge 0) \equiv (X_t' : t \ge 0)$.

We adopt in this section the notation introduced in Section \ref{sec1}.

\subsection*{Donsker-Varadhan rate functional}

We start with the Donsker-Varadhan rate functional, especially with two elementary identities. By Lemma A.7 and equation
(A.14) in \cite{l-gamma},
\begin{equation}
\label{ap-02}
\ms I(\delta_z) = \lambda (z) \quad \text{for all}\quad z\in V\;.
\end{equation}

Fix $x\neq y\in V$, and let
$\mu_\theta = \theta \delta_x + (1-\theta) \delta_y$, $0<\theta<1$. We
claim that
\begin{equation}
\label{ap-01}
\ms I(\mu_\theta) \,=\, 
\theta \, \ms I(\delta_x)  \,+\, (1-\theta)\, \, \ms I(\delta_y)
\,-\, 2\, \sqrt{R(x,y)\, R(y,x)}\,
\sqrt{\theta (1-\theta)} \;.
\end{equation}
Indeed, if $R(x,y)=0$ or $R(y,x)=0$, by Lemma A.7 and equation (A.14)
in \cite{l-gamma},
$\ms I(\mu_\theta) = \theta \, \ms I(\delta_x) + (1-\theta) \, \ms
I(\delta_y)$. On the other hand, if $R(x,y)\, R(y,x)>0$, by the same
result,
\begin{equation*}
\ms I(\mu_\theta) \,=\, \ms I_{x,y} (\mu_\theta)
\,+\, \theta \,[\, \lambda (x) - R(x,y)\,]\,+\, (1-\theta)\,
\,[\, \lambda (y) - R(y,x)\,]\;.
\end{equation*}
In this formula, $\ms I_{x,y}$ represents the DV large deviations rate
functional of the $\{x,y\}$-valued Markov chain which jumps from $x$
to $y$ at rate $R(x,y)$ and from $y$ to $x$ at rate $R(y,x)$.  Since a
two-state chain is always reversible, an elementary computation
yields that
\begin{equation*}
\ms I_{x,y}(\mu_\theta)  \,=\, \theta\, R(x,y) \,+\,
(1-\theta) \, R(y,x) \,-\, 2\, \sqrt{R(x,y)\, R(y,x)}\,
\sqrt{\theta (1-\theta)}\;.
\end{equation*}
Adding the two previous expressions yields that
\begin{equation*}
\ms I(\mu_\theta) \, =\, 
\theta \, \lambda (x) \,+\, (1-\theta)\, \,\lambda (y)
\,-\, 2\, \sqrt{R(x,y)\, R(y,x)}\,
\sqrt{\theta (1-\theta)} \;.
\end{equation*}
Since $\lambda (z) = \ms I(\delta_z)$, Claim \eqref{ap-01} is proved.
We summarize in Lemma \ref{ap-l3} the previous obsevations.

\begin{lemma}
\label{ap-l3}
One can derive from $\ms I(\cdot)$ the values of $\lambda (z)$,
$z\in V$, and $R(x,y) \, R(y,x)$ for all $y\neq x\in V$.
\end{lemma}

The next result shows that we may also recover from $\ms I(\cdot)$ the stationary profile.
Suppose that $V_1 , \dots, V_n$ are the closed irreducible classes of the chain,
and denote by $\pi_j$, $1\le j\le n$, the invariant probability measure on $V_j$.

\begin{lemma}
\label{ap-l1}
One can derive from $\ms I(\cdot)$ the measures $\pi_j$ for all $1\le j\le n$.
\end{lemma}

\begin{proof}
By \cite[Lemma A.8]{l-gamma}, $\ms I^{-1}(0)$ equals the set of all stationary states,
which are exactly the non-negative linear combinations of $\pi_j$, $1\le j\le n$.
Thus, from the DV rate functional $\ms I$, we can determine each stationary measure
$\pi_j$ on $V_j$.
\end{proof}

\begin{corollary}
\label{ap-l02}
Assume that the chain is reversible.
Then, it is possible to recover the generator $\ms L$ of
the Markov chain from the rate functional $\ms I$.
\end{corollary}

\begin{proof}
By reversibility, all elements of $V$ are recurrent, thus $V=V_1 \cup \cdots \cup V_n$.
Fix a collection $V_j$ and $x,y \in V_j$. By Lemma \ref{ap-l3} we recover from $\ms I$
the value of $R(x,y) \, R(y,x)$, and by Lemma \ref{ap-l1} we recover the values of $\pi_j(x)$ and $\pi_j(y)$.
By the detailed balance condition, $\pi_j(x) \, R(x,y) = \pi_j(y) \, R(y,x)$, thus
\[
R(x,y)^2 = R(x,y) \, \frac{\pi_j(y) \, R(y,x)}{\pi_j(x)} = \frac{\pi_j(y)}{\pi_j(x)} \, R(x,y) \, R(y,x) \;,
\]
which implies that we recover the value of $R(x,y)$ and similarly $R(y,x)$. This concludes the proof.
\end{proof}

\begin{example}\label{ex01}
This example illustrates the fact that we can not do
better, that the DV large deviations rate functional does not
characterise the dynamics if the chain is not reversible.
Consider the irreducible dynamics on $V = \{a,b,c\}$ whose jump rates
are given by $R(a,b) = R(b,c) = R(c,a)=1$ and $0$ otherwise. An elementary computation
yields that
\begin{equation*}
\ms I(\mu) \,=\, 1\,-\, 3 \, (\mu_a\, \mu_b\, \mu_c)^{1/3}\;.
\end{equation*}
Hence, this dynamics has the same DV rate functional as the one
where the jumps occur in the opposite cyclic order.
\end{example}

\subsection*{Bertini-Faggionato-Gabrielli rate functional}

For a function $H \colon V\to \bb R$, denote by
$\ms J_H \colon \ms P(V)\to \bb R$ the functional given by
\begin{equation*}
{\color{blue} \ms J_H (\mu)} \, := \, -\,\int_V e^{-H}\, \ms L e^H\, {\rm d}\,\mu
\,=\,\sum_{x, y\in V} \mu(x)\, R(x,y)\,
\big[\, 1 - e^{H (y) - H (x)}\,\big] \;.
\end{equation*}
Therefore,
\begin{equation*}
\ms I (\mu) \,=\, \sup_{H} \, \ms J_H (\mu) \;,
\end{equation*}
where the supremum is carried over all functions $H \colon V\to \bb
R$.

Denote by $R_{H}(x,y)$ the jump rates of the dynamics tilted by a
function $H\colon V\to \bb R$:
$\color{blue} R_{H}(x,y) := R(x,y) \, \exp \{H(y) - H(x)\}$. Let
$\color{blue} \ms L_{H}$ be the generator associated to the rates
$R_{H}$.  Next results are Lemmata A.2 and A.3 in\cite{l-gamma}.

\begin{lemma}
\label{ap-l10}
A measure $\mu\in\ms P(V)$ is a stationary state for the Markov chain
induced by the generator $\ms L_{H}$ if and only if
\begin{equation}
\label{ap-12}
\ms I(\mu) \,=\, \ms J_H(\mu)\;.
\end{equation}
\end{lemma}

\begin{lemma}
\label{ap-l11}
Suppose that the Markov chain is irreducible. Fix a measure
$\mu \in\ms P(V)$ such that $\mu(x)>0$ for all $x\in V$.  Then, there
exists a function $H\colon V\to \bb R$, denoted by
$\color{blue} H_\mu$, such that
$\ms I(\mu) \,=\, \ms J_{H_\mu} (\mu)$.  The function $H_\mu$ is
unique, up to an additive constant.
\end{lemma}

% there exists a
% unique function, denoted by $H_\nu$, such that $H_\nu(x_0)=0$ and
% $\ms I(\nu) = \ms J_{H_\nu}(\nu)$. 

% such that $\mu(x)>0$ for all $x\in V$, there exists a
% function $H\colon V\to \bb R$, unique up to an additive constant and
% denoted here by $H_\mu$, such that
% \begin{equation}
% \label{ap-07}
% \ms I(\mu) \,=\, \sum_{x, y\in V} \mu(x)\, R(x,y)\,
% \big[\, 1 - e^{H(y) - H(x)}\,\big] \;.
% \end{equation}

% uuu

% For
% $\mu\in \ms P(V)$, denote by $\color{blue} H_\mu \colon V\to \bb R$
% the optimal function in the variational problem

% By \cite[Lemma A.3]{l-gamma}, if the chain is irreducible and
% $\mu(x)>0$ for all $x\in V$, $H_\mu$ is uniquely defined up to an
% additive constant.

Recall from \eqref{02} the definition of the BFG large deviations rate
functional $I\colon \ms P(V) \times \mf F_E \to [0,+\infty]$, and let
$\color{blue} \ms P_+(V) := \{ \mu\in\ms P(V) : \mu(x)>0\,\; \forall
x\in V\}$.

\begin{lemma}
\label{ap-l8}
Suppose that the Markov chain is irreducible. For all $\mu\in \ms P_+(V)$,
there exists a unique current $J_\mu^* \in \mf F^{\rm div}_E$ which minimises
$I(\mu, \cdot)$, which is given by
$J_\mu^*(x,y) = \mu(x) \, R(x,y) \, \exp\{H_\mu(y) - H_\mu(x)\}$.
\end{lemma}

\begin{proof}
By \cite[Theorem 1.6]{bfg}, the existence of a unique current $J_\mu^* \in \mf F_E^{\rm div}$ minimising $I(\mu,\cdot)$ is guaranteed.
As already claimed in \eqref{gn}, since $\mu(x)>0$ for all $x\in V$, \cite[Eq. (2.9)]{bfg} guarantees the existence of a function $H:V \to \bb R$ such that
\begin{equation}
\log \frac{J_\mu^*(x,y)}{ \mu(x) \, R(x,y) } \,
= \, H (y) \, - \, H (x) \quad \text{for each } (x,y) \in E \; .
\end{equation}
It remains to prove that $H-H_\mu$ is a constant function. By \cite[Eq. (2.17)]{bfg},
\begin{equation}
\ms I (\mu) \, = \, I(\mu,J_\mu^*) = \sum_{(x,y) \in E} \mu(x) \, R(x,y)  \, \big( \, 1 \, - \, e^{H(y) \, - \, H(x)} \, \big) \, = \, \ms I_H(\mu) \; .
\end{equation}
Thus, by the uniqueness statement in Lemma \ref{ap-l11}, $H-H_\mu$ is constant, which completes the proof.
\end{proof}

\begin{corollary}
\label{ap-l9}
Assume that all elements of $V$ are recurrent. Then, we
may recover from the BFG rate functional
$I\colon \ms P(V) \times \mf F_E \to [0,+\infty]$ the jump rates $R$.
%Moreover, for all measure
%$\mu\in \ms P_+(V)$, we may also obtain from $I$ the function $H_\mu$.
\end{corollary}

\begin{proof}
Since all states are recurrent, we have $V=V_1 \cup \cdots \cup V_n$ where $V_j$'s are pairwise isolated, i.e., $R(x,y)=0$ for all $x\in V_j$, $y\in V_k$ such that $j \ne k$. For each $j$, denote by $E_j$ the set of edges whose both endpoints belong to $V_j$, such that $E=E_1 \cup \cdots \cup E_j$.

First, we claim that we can recover each $I_j \colon \ms P(V_j) \times \mf F_{E_j} \to [0,+\infty]$, $1\le j\le n$, which is the BFG rate functional for the chain restricted to $V_j$. Indeed, given any pair $(\mu_j , J_j) \in \ms P(V_j) \times \mf F_{E_j}$, extend it to $\ms P(V) \times \mf F_E$ by declaring $\mu_j$ and $J_j$ as $0$ outside $V_j$ and $E_j$, respectively. By \eqref{02}, if $\mu_j \notin \mf F_{E_j}^{\rm div}$ then $I_j(\mu_j,J_j)=+\infty$, and if $\mu_j \in \mf F_{E_j}^{\rm div}$ then
\[
I(\mu_j,J_j) \, = \, I_j(\mu_j,J_j) \, + \, \sum_{k \ne j} \sum_{(x,y) \in E_k} \Phi (0,0) = I_j (\mu_j ,J_j) \; .
\]
Thus, $I_j(\mu_j,J_j) = I(\mu_j,J_j)$ in any case, and $I_j$ is recovered from $I$.

Now, fix a collection $V_j$ and consider the restricted irreducible chain on $V_j$. Since $\ms I$ is recovered from $I$ by $\ms I(\mu) = \inf_J I(\mu,J)$, by Lemma \ref{ap-l1}, we readily recover the stationary profile $\pi_j \in \ms P_+(V_j)$.
Then, Lemma \ref{ap-l8} implies that $I^{-1}(0)$ equals $\{ (\pi_j,J_{\pi_j}^*) \}$ with
$J_{\pi_j}^*(x,y) = \pi_j(x) \, R(x,y) \, \exp\{H_{\pi_j}(y) - H_{\pi_j}(x)\}$, where $H_{\pi_j}$ is constant by Lemma \ref{ap-l11} since $\ms I(\pi_j) = 0 = \ms I_0 (\pi_j)$ owing to the stationarity of $\pi_j$. This implies that $J_{\pi_j}^*(x,y) = \pi_j(x) \, R(x,y)$, thus, we may recover $R(x,y)$ for all $(x,y) \in E_j$ due to the fact that $\pi_j (x) > 0$ for all $x \in V_j$. This concludes the proof of the corollary.
%For the second statement, fix $\mu \in \ms P_+(V)$. By Lemma \ref{ap-l8}, for each $1\le j\le n$,
%from $I$ we recover $J_\mu^*$, thus recover $H_\mu$ up to an additive constant.
\end{proof}

\begin{example}\label{ex02}
We claim that Corollary \ref{ap-l9} is optimal, in the sense that the BFG large deviations rate functional does not characterise the dynamics if there exists a transient state.

Consider two three-state Markov chains on
$V = \{a,b,c\}$. In both, the Markov chain jumps back and forth from
$b$ to $c$ at rate $1$. In the first it jumps from $a$ to $b$ at rate
$2$, while in the second it jumps from $a$ to $c$ at rate $2$.
Denote by $I$ and $I'$ the respective BFG rate functionals.
Then, we have
\begin{equation*}
I(\mu,J) \, = \, \begin{cases}
\Phi (J_{ab}, 2\mu_a) \, + \, \Phi (J_{bc}, \mu_b) \, + \, \Phi (J_{cb}, \mu_c) & \text{if}\;\; J \in \mf F_E^{\rm div} \; ,\\
+\infty & \text{otherwise}\;.
\end{cases}
\end{equation*}
Since $J \in \mf F_E^{\rm div}$ if and only if $J_{ab}=0$ and $J_{bc}=J_{cb}$, we may reformulate
\begin{equation}
I(\mu,J) \, = \, \begin{cases}
2\mu_a \, + \, \Phi (J_{bc}, \mu_b) \, + \, \Phi (J_{bc}, \mu_c) & \text{if}\;\; J \in \mf F_E^{\rm div} \; ,\\
+\infty & \text{otherwise}\;.
\end{cases}
\end{equation}
The right-hand side above is symmetric in $b$ and $c$, which implies that $I = I'$.
\end{example}

We complete this subsection with a result on the map
$\mu \mapsto H_\mu$. Choose $x_0\in V$ and set $H_\mu(x_0) =0$ for
all $\mu$. With this convention, $H_\mu$ is unique by Lemma \ref{ap-l11}.

\begin{lemma}
\label{ap-l12}
Under the hypotheses of Lemma \ref{ap-l8}, and with the previous
convention, the map $\mu \mapsto H_\mu$ is a bijection.
\end{lemma}

\begin{proof}
The map is surjective. Indeed, fix a function $H$ such that
$H(x_0)=0$. By the hypotheses of Lemma \ref{ap-l8}, the original
dynamics is irreducible. Hence, the one induced by the generator
$\ms L_H$ is also irreducible. It admits, therefore, a unique
stationary state, denoted by $\nu$. By Lemma \ref{ap-l10},
$\ms I(\nu) = \ms J_H(\nu)$.  On the other hand, by irreducibility
$\nu\in \ms P_+(V)$. Thus, by Lemma \ref{ap-l11}, there exists a
unique function, denoted by $H_\nu$, such that $H_\nu(x_0)=0$ and
$\ms I(\nu) = \ms J_{H_\nu}(\nu)$. By uniqueness, $H=H_\nu$, which
proves that the map is surjective.

The map is clearly injective. Indeed, suppose that $H_\mu=H_\nu$. By
Lemma \ref{ap-l11}, $\ms I(\mu) = \ms J_{H_\mu}(\mu)$. Thus, by Lemma
\ref{ap-l10} $\mu$ is stationary for the Markov chain induced by the
generator $\ms L_{H_\mu}$. Since $\ms L_{H_\mu} = \ms L_{H_\nu}$,
$\mu$ is also stationary for the Markov chain induced by the generator
$\ms L_{H_\nu}$. By the uniqueness of stationary states, $\nu=\mu$, as
claimed.
\end{proof}

\subsection*{Functional derivative}

In this subsection, we assume that the chain is irreducible, such that
there exists a unique positive stationary state, say $\pi \in \ms P_+(V)$.
By \cite[Lemma A.8]{l-gamma}, $\ms I^{-1}(0) = \pi$.

For a measure $\mu\in \ms P(V)$, denote by $L^2(\mu)$ the space of
$\mu$-square summable functions $f\colon V \to \bb R$, and by
$\< \cdot \,,\, \cdot \>_{\mu}$ the scalar product in $L^2(\mu)$:
\begin{equation*}
\< f, g \>_{\mu} \,=\, \sum_{x\in V} \mu(x)\, f(x)\, g(x)\;, \quad
f\,,\, g \,\in\, L^2(\mu) \;.
\end{equation*}
Let $\color{blue} \ms L_{H_\mu}^*$ be the adjoint of $\ms L_{H_\mu}$
in $L^2(\mu)$, and $\ms L^s_{H_\mu}$ the symmetric part of
$\ms L_{H_\mu}$,
$\color{blue} \ms L^s_{H_\mu} := (1/2) (\ms L_{H_\mu} + \ms
L^*_{H_\mu})$.

% \begin{equation}
% \label{ap-07}
% \ms I(\nu) \,=\, \ms J_{H_\mu}(\mu) \;.
% \end{equation}
% By \cite[Lemmata A.3]{l-gamma}, $\nu$ is stationary for the
% Markov chain induced by the rates $R_{H_\nu}$.

% Thus $\nu$ is stationary for Markov chains induced by the rates $R_H$
% and $R_{H_\nu}$.

% Recall from Lemmata A.2 and A.3 in \cite{l-gamma} that given a measure
% $\mu\in\ms P(V)$ such that $\mu(x)>0$ for all $x\in V$, there exists a
% function $H\colon V\to \bb R$, unique up to an additive constant and
% denoted here by $H_\mu$, such that
% \begin{equation}
% \label{ap-07}
% \ms I(\mu) \,=\, \sum_{x, y\in V} \mu(x)\, R(x,y)\,
% \big[\, 1 - e^{H(y) - H(x)}\,\big] \;.
% \end{equation}
% Moreover, $\mu$ is the stationary state of the dynamics tilted by
% $H$.

Let $\ms G\colon \ms P(V) \times L^\infty(V) \to \bb R$ be the
functional given by
\begin{equation*}
{\color{blue} \ms G(\mu, H)} \,:=\, \sum_{x, y\in V} \mu(x)\, R(x,y)\,
\big[\, 1 - e^{H(y) \, - \, H(x)}\,\big] \;,
\end{equation*}
so that
\begin{equation}
\label{ap-10}
\ms I (\mu) \,=\, \sup_{H} \, \ms G(\mu, H)  \;.
\end{equation}
In the next lemma, we compute the functional derivative of $\ms I$
defined by
\begin{equation*}
{\color{blue} \frac{\delta \, \ms I}{\delta \mu}  (\mu; \nu)} \,:=\,
\lim_{\epsilon\to 0} \frac{1}{\epsilon}\, \big\{ \, \ms I(\mu + \epsilon\nu)
- \ms I(\mu) \,\big\} \;.
\end{equation*}
For $\mu + \epsilon\nu$ to be a probability measure, we need to assume
that $\sum_x \nu(x) =0$ and that $\nu(y) \ge 0$ for all $y\in V$ such that
$\mu(y)=0$.

In view of \eqref{ap-10} and Lemma \ref{ap-l11}, $\ms I (\mu) \,=\, \ms G(\mu, H_\mu)$ so
that, formally,
\begin{equation*}
\frac{\delta \, \ms I}{\delta \mu} (\mu; \nu) \,=\,
\frac{\delta \, \ms G}{\delta \mu} (\mu, H_\mu; \nu) \,+\,
\frac{\delta \, \ms G}{\delta H} (\mu, H_\mu; \nu) \,
\frac{\delta H_\mu}{\delta \mu} \;.
\end{equation*}
As $H_\mu$ is a critical point of $\ms G (\mu, \cdot )$, $(\delta
\, \ms G/ \delta H) (\mu, H_\mu; \cdot ) =0$ and
\begin{equation*}
\frac{\delta \, \ms I}{\delta \mu} (\mu; \nu) \,=\,
\frac{\delta \, \ms G}{\delta \mu} (\mu, H_\mu; \nu ) \,=\,
\sum_{x, y\in V} \nu(x)\, R(x,y)\, \big[\, 1 \, - \, e^{H_\mu(y) \, - \, H_\mu
(x)}\,\big]\;.
\end{equation*}
Next result makes this formal argument rigorous.

\begin{lemma}
\label{ap-l6}
Fix a probability measure $\mu\in \ms P(V)$ such that $\mu(x)>0$ for
all $x\in V$, and $f \colon V \to \bb R$ such that $E_\mu[f]=0$.  Let
$\mu_\epsilon = \mu +\epsilon \nu = \mu (1+\epsilon f)$,
$\epsilon>0$. Then,
\begin{equation*}
\frac{\delta \, \ms I}{\delta \mu}  (\mu; \nu) \,=\,
\sum_{x, y\in V} \nu(x)\, R(x,y)\, \big[\, 1 \, - \, e^{H_\mu(y) \, - \, H_\mu
(x)}\,\big]
\;.
\end{equation*}
In particular, as $H_\pi =0$, $(\delta/\delta \mu)  \, \ms I (\pi; \nu) = 0$ for all
$\nu$. 
\end{lemma}

\begin{proof}
By definition, $\mu_\epsilon$ is a probability measure for $\epsilon$
sufficiently small. Let $H_\epsilon\colon V\to \bb R$ be the function
associated to $\mu_\epsilon$, $H_\epsilon = H_{\mu_\epsilon}$
introduced in Lemma \ref{ap-l11}.

\smallskip\noindent{\sl Claim A:} The sequence of measures
$H_\epsilon$ converges to $H_\mu$ as $\epsilon\to 0$.
\smallskip

Indeed, by equation (A.8) in \cite{l-gamma}, the sequence
$H_\epsilon(y) - H_\epsilon (x)$ is uniformly (in $x\neq y\in V$,
$\epsilon>0$) bounded.  Since $\mu_\epsilon$ is the stationary state
for the dynamics tilted by $H_\epsilon$,
\begin{equation*}
\mu_\epsilon(x) \sum_{y\in V} R(x,y)\, e^{H_\epsilon(y) \,-\, H_\epsilon(x)}
\,=\,
\sum_{y\in V} \mu_\epsilon(y)  \, R(y,x)\, e^{H_\epsilon(x) \,-\, H_\epsilon(y)}
\end{equation*}
for all $x\in V$. Let $H$ be a limit point of the sequence
$H_\epsilon$. Since $\mu_\epsilon \to \mu$, letting $\epsilon\to 0$
yields that $\mu$ is the stationary state of the tilted dynamics by
$H$.  Since $\mu(x)>0$ for all $x\in V$, by the uniqueness result
stated in \cite[Lemma A.3]{l-gamma}, $H = H_\mu$, proving the claim.

\smallskip\noindent{\sl Claim B:}
Let $G_\epsilon\colon V \to \bb R$ be defined as
$H_\epsilon(x) = H_\mu(x) + \epsilon \, G_\epsilon(x)$.
Then, 
\begin{align*}
\lim_{\epsilon\to 0}  G_\epsilon
\,=\, (1/2)\,  (\ms L^s_{H_\mu})^{-1}\, \ms  L^*_{H_\mu} f\;.
\end{align*}
\smallskip

We first show that the sequence $G_\epsilon (y) -G_\epsilon(x)$ is
uniformly bounded. By the previous claim,
$\epsilon \, G_\epsilon(x) \to 0$ for all $x\in V$.  Since $\mu_\epsilon$
is the stationary state of the tilted dynamics by $H_\epsilon$,
\begin{equation*}
\sum_{x, y \in V} \mu (x)\, [1+ \epsilon \, f (x)]\,
R_{H_\mu}(x,y)\, e^{ \epsilon \, [G_\epsilon(y) \,-\, G_\epsilon(x)] } \,\{g(y) \,-\, g(x)\}\,=\, 0
\end{equation*}
for every function $g\colon V\to \bb R$. Let $A_\epsilon(x,y) \, := \, 1$ if $G_\epsilon(x) = G_\epsilon(y)$ and
\begin{equation*}
A_\epsilon (x,y) \,:=\, \frac{ e^{ \epsilon \, [G_\epsilon(y)
\,-\, G_\epsilon(x)] } \,-\, 1}{\epsilon \, [G_\epsilon(y)
\,-\, G_\epsilon(x)] }\;,
\end{equation*}
if $G_\epsilon(x) \ne G_\epsilon(y)$.
As $\epsilon \, G_\epsilon \to 0$,
$A_\epsilon (x,y) = 1 + a_\epsilon (x,y)$, where
$a_\epsilon (x,y)\to 0$ for all $y\neq x\in V$.  Since $\mu$ is the
stationary state for the dynamics tilted by $H_\mu$,
\begin{equation}
\label{ap-08}
\begin{aligned}
0\, =&\, \sum_{x, y\in V} \mu_\epsilon (x)\,
R_{H_\mu}(x,y)\, A_\epsilon (x,y)\,
[G_\epsilon (y) -G_\epsilon(x)] \,\{g(y) - g(x)\} 
\\
\,& +\, \sum_{x, y\in V} \mu(x)\, f (x)\,
R_{H_\mu}(x,y)\, \{g(y) - g(x)\}
\end{aligned}
\end{equation}
for every function $g\colon V\to \bb R$. 

Fix $y_0\in V$, and set $g = \delta_{y_0}$. With this choice, the
previous equation becomes
\begin{align*}
& \sum_{x\in V}
\big\{ \, \mu_\epsilon (x) \, R_{H_\mu}(x,y_0)\, A_\epsilon (x,y_0)\,
+\, \mu_\epsilon (y_0) \,  R_{H_\mu}(y_0,x)\, A_\epsilon (y_0,x)\,\big\}\,
[G_\epsilon(x) \,-\, G_\epsilon(y_0)]
\\
&\, =\,
\sum_{x\in V} \big\{ \, \mu (x)\, f (x) \, R_{H_\mu}(x,y_0)\,-\,
\mu (y_0)\, f (y_0) \, R_{H_\mu}(y_0,x)\,\big\}
\;.
\end{align*}
Choose $y_0$ as the point in $V$ where $G_\epsilon$ attains its
minimum. As $A_\epsilon = 1 + a_\epsilon$ and $\mu_\epsilon = \mu + \epsilon \nu$,
where $a_\epsilon\to 0$ as $\epsilon \to 0$ and $\mu$ is bounded away from $0$,
%$A_\epsilon\ge 1/2$ for sufficiently small $\epsilon$.
All terms in the first sum are non-negative, and strictly positive if
$R(x,y_0) + R(y_0,x)>0$. Therefore, there exists a positive constant
$C_0$, independent of $\epsilon$ such that
$0\le G_\epsilon(x) -G_\epsilon(y_0) \le C_0$ for all $x\in V$ such
that $R(x,y_0) + R(y_0,x)>0$.

Denote by $N(y)$ the neighbourhood of $y$, the set of all points
$x\in V$ such that $R(x,y) + R(y,x)>0$.  Let $y_1$ be the minimum in
the remaining sites:
$G_\epsilon (y_1) = \min \{G_\epsilon (x) : x \neq y_0\}$. If
$y_1\not\in N(y_0)$, all terms in the first sum (with $y_1$ replacing
$y_0$) are positive and we can repeat the same argument. If not, the
term $x=y_0$ is negative, but $G_\epsilon(y_1) -G_\epsilon(y_0)$ is
bounded by a constant independent of $\epsilon$ by the previous
paragraph. We may move this term to the righ-hand side of the equality
and obtain a uniform bound for $|\, G_\epsilon(y_1) -G_\epsilon(x)\,|$
for all $x\in N(y_1)$.

If $N(y_0) \cap N(y_1) \neq \varnothing$, there is a uniform in
$\epsilon$ bound for $|\, G_\epsilon(x) -G_\epsilon(x')\,|$ for
$x,x'\in N(y_0) \cup N(y_1)$. Since the chain is irreducible, we may
proceed with this inductive argument to obtain a uniform bound for
$|\, G_\epsilon(x) -G_\epsilon(y)\,|$ for all $x,y\in V$.

Let $G$ be a limit point of the sequence $G_\epsilon$. By
\eqref{ap-08},
\begin{align*}
0\, =&\, \sum_{x, y\in V} \mu (x)\,
R_{H_\mu}(x,y)\,  [G (y) \,-\, G (x)] \,\{g(y) \,-\, g(x)\} 
\\
\,& +\, \sum_{x, y\in V} \mu(x)\, f (x)\,
R_{H_\mu}(x,y)\, \{g(y) \,-\, g(x)\}\;.
\end{align*}
This equation can be rewritten as
\begin{align*}
2\, \< \ms L^s_{H_\mu} G\,,\, g\>_{\mu} \,=\, \< f\,,\, \ms
L_{H_\mu} g\>_{\mu}\; .
\end{align*}
Since the previous equation holds for all $g$, Claim B is
proved.\smallskip

We may now complete the proof of the lemma. By \cite[Lemma A.3]{l-gamma},
\begin{equation}
\label{ap-03}
\ms I(\mu_\epsilon) \,=\, \sum_{x,y\in V} \mu_\epsilon(x)\, R(x,y)\,
\big[\, 1 \,-\, e^{ H_\epsilon(y) \,-\, H_\epsilon(x)}\,\big]\;.
\end{equation}
Replace $\mu_\epsilon(x)$ by $\mu (x) \, (1 + \epsilon f(x))$,
and expand in $\epsilon$ to obtain that
\begin{equation*}
\begin{aligned}
\frac{1}{\epsilon}\, 
\big\{ \ms I(\mu_\epsilon) \,-\, \ms I(\mu) \big\} \, =\,
& -\,  \sum_{x,y\in V} \mu (x)\, R_{H_\mu} (x,y)\,
[G_\epsilon(y) -G_\epsilon(x)] \, A_\epsilon(x,y)
\\
& + \, \sum_{x,y\in V} \mu (x) \, f (x)\, R(x,y)\,
\big[\, 1 \,-\, e^{ H_\epsilon(y) \,-\, H_\epsilon(x)}\,\big]
\,+\, O(\epsilon) \;.
\end{aligned}
\end{equation*}
The first line vanishes as $\epsilon \to 0$ because $A_\epsilon \to 1$ and
$\mu$ is the stationary state for the
tilted dynamics. Via Claim A, this completes the proof of the lemma.
\end{proof}

\begin{remark}
 In the previous proof, we have shown that
\begin{equation}
\label{ap-09}
\frac{\delta H_\mu}{\delta \mu} (\mu; \nu)  \,=\,
(1/2)\,  (\ms L^s_{H_\mu})^{-1}\, \ms  L^*_{H_\mu} f\;,
\end{equation}
where $f = {\rm d}\nu/{\rm d}\mu$.
\end{remark}

We turn to the second derivative of the rate functional.

\begin{proposition}
\label{ap-l7}
Fix a probability measure $\mu\in \ms P(V)$ such that $\mu(x)>0$ for
all $x\in V$, and two signed measures $\nu_i \colon V \to \bb R$,
$i=1,2$, such that $\sum_x \nu_i(x) =0$. Then,
\begin{equation*}
\begin{aligned}
\frac{\delta^2 \, \ms I}{\delta \mu^2}  (\mu; \nu_1, \nu_2) \, & :=\,
\lim_{\epsilon\to 0} \frac{1}{\epsilon}\, \Big\{ \,
\frac{\delta \, \ms I}{\delta \mu} (\mu + \epsilon\nu_2;\nu_1)
\,-\, \frac{\delta \, \ms I}{\delta \mu} (\mu; \nu_1) \,\Big\}
\\
\, &=\, \frac{1}{2}\, \< f_1 \,,\, \ms L_{H_\mu} \, ( - \ms
L^s_{H_\mu})^{-1} \, \ms L^*_{H_\mu} f_2 \>_\mu
\;,
\end{aligned}
\end{equation*}
where $f_i = {\rm d}\nu_i/{\rm d}\mu$.
\end{proposition}

\begin{proof}
This results follows from the explicit formula for the functional
derivative $\delta \, \ms I/\delta \mu$ derived in the previous lemma
and from formula \eqref{ap-09}.
\end{proof}

Recall that the second derivative of the large deviations principle
rate functional for i.i.d. random variables yields the inverse of the
variance. A similar statement holds in the context of finite state
Markov chains.

Fix a function $f\colon V\to \bb R$ which has $\pi$-mean zero.  Recall
from Theorem 2.7 and Corollary 2.11 in \cite{klo} that the
asymptotic variance of the central limit theorem for
\begin{equation*}
\frac{1}{\sqrt{t}} \int_0^t f(X_s)\, {\rm d}s 
\end{equation*}
under the stationary state $\pi$ is equal to
\begin{equation*}
\sigma^2(f)\,:=\,
2\, \< \ms L^{-1} f \,,\, (-\ms L^s) \, \ms L^{-1} f\>_\pi\;.
\end{equation*}
We claim that
\begin{equation}
\label{ap-11}
\frac{\delta^2 \, \ms I}{\delta \mu^2}  (\pi; \nu, \nu)\,=\,
\sup_{h} \big\{\, 2\<f,h\>_{\pi} \,-\, \sigma^2 (h)\,\big\}\;,
\end{equation}
where the supremum is carried over all $\pi$-mean zero functions $h$
and $f = {\rm d}\nu/{\rm d}\pi$.
Indeed, by Proposition \ref{ap-l7} and the variational formula for the norm
$\< \cdot, (-\ms L^s )^{-1} \cdot\>_{\pi}^{1/2}$, the left-hand side is equal to 
\begin{equation*}
\begin{aligned}
\frac{1}{2}\, \sup_{g} \big\{\, 2\, \<\ms L^* f,g\>_{\pi} \,-\,
\< g , (-\ms L^s) g \>_{\pi}\,\big\}
\,=\, \frac{1}{2}\, \sup_{g} \big\{\, 2\, \< f, \ms L g\>_{\pi} \,-\,
\< g , (-\ms L) g \>_{\pi}\,\big\}\;,
\end{aligned}
\end{equation*}
where the supremum is carried over all $\pi$-mean zero functions
$g$. We used here the definition of the adjoint and the symmetric part
of the generator $\ms L$. As the process is irreducible, the equation
$\ms L g =h$ has a solution for all $\pi$-mean zero functions
$h$. Performing this change of variables, the previous equation
becomes
\begin{equation*}
\begin{aligned}
& \frac{1}{2}\, \sup_{h} \big\{\, 2\, \< f, h\>_{\pi} \,-\,
\< \ms L^{-1} h , (-\ms L) \ms L^{-1} h \>_{\pi}\,\big\}
\\
&\quad \,=\,
\frac{1}{2}\, \sup_{h} \big\{\, 2\, \< f, h\>_{\pi} \,-\,
\< \ms L^{-1} h , (-\ms L^s) \ms L^{-1} h \>_{\pi}\,\big\}
\;,
\end{aligned}
\end{equation*}
where the supremum is carried over all $\pi$-mean zero functions
$h$. In view of the definition of $\sigma^2(h)$, the last line is
equal to
\begin{equation*}
\frac{1}{2}\, \sup_{h} \big\{\, 2\, \< f, h\>_{\pi} \,-\,
\frac{1}{2}\, \sigma^2(h) \,\big\}
\,=\, \sup_{h} \big\{\, 2\, \< f, h\>_{\pi} \,-\, \sigma^2(h) \,\big\}\;,
\end{equation*}
where we performed the change of variables $h'=(1/2) h$ in the last
step. This proves \eqref{ap-11}. \smallskip

By Lemma \ref{ap-l12}, there is a one-to-one correspondance between
measures $\mu\in\ms P_+(V)$ and functions $H\colon V\to \bb R$ such
that $H(x_0)=0$.  By Lemma \ref{ap-l8}, for each
$\mu\in\ms P_+(V)$ we may recover the function $H_\mu$ through the BFG
functional $I$. Therefore, we may consider the DV rate functional as a
functional defined on functions $H\colon V\to \bb R$ instead of
measures $\mu\in \ms P_+(V)$.  In particular, we may compute the
derivative of $\ms I$ with respect to $H$. This is the content of the
next lemma.

\begin{lemma}
\label{ap-l4}
Fix a function $H\colon V\to \bb R$. Let $\pi_\epsilon$, $\epsilon>0$,
be the stationary state of the Markov chain tilted by $\epsilon H$
(the jump rates of this chain are given by
$R_{\epsilon H}(x,y) = R(x,y) \, \exp \{\epsilon [H(y) -
H(x)]\}$). Then,
\begin{equation*}
\lim_{\epsilon\to 0}
\frac{1}{\epsilon^2}\, \ms I(\pi_\epsilon) \, =\, \< (-\ms L^s) H, H\>_\pi \;.
\end{equation*}
\end{lemma}

\begin{proof}
We claim that the sequence of measures $(\pi_\epsilon : \epsilon>0)$
converges to $\pi$ as $\epsilon\to 0$. Indeed, it is a bounded
sequence. Since it is invariant for the tilted dynamics,
\begin{equation*}
\pi_\epsilon(x) \sum_{y\in V} R(x,y)\, e^{\epsilon \, [H(y) \,-\, H(x)]}
\,=\,
\sum_{y\in V} \pi_\epsilon(y)  \, R(y,x)\, e^{\epsilon \,[H(x) \,-\, H(y)]}
\end{equation*}
for all $x\in V$. Letting $\epsilon\to 0$ yields that $\pi$ is the
unique limit point, which proves the claim.

Let $f_\epsilon\colon V \to \bb R$ be defined as
$\pi_\epsilon(x) = \pi(x) \,[\, 1 + \epsilon f_\epsilon(x)\,]$. Note
that $E_\pi[f_\epsilon]=0$ for every $\epsilon>0$. Moreover, from the
previous claim, $\epsilon f_\epsilon(x) \to 0$ for all $x\in V$.

Since $\pi_\epsilon$ is the stationary state of the tilted dynamics,
\begin{equation*}
\sum_{x, y \in V} \pi_\epsilon (x)\,
R(x,y)\, e^{\epsilon \,[H(y) \,-\, H(x)]} \,\{g(y) \,-\, g(x)\}\,=\, 0
\end{equation*}
for every function $g\colon V\to \bb R$. Replacing $\pi_\epsilon$ by
$\pi(x) \,[\, 1 + \epsilon f_\epsilon(x)\,]$, as $\pi$ is the
stationary state, yields that
\begin{equation}
\label{ap-05}
\begin{aligned}
0\, =&\, \sum_{x, y\in V} \pi (x)\,
R(x,y)\, [H(y) -H(x)] \,\{g(y) - g(x)\} \, +\,  O(\epsilon)
\\
\,& +\, \sum_{x, y\in V} \pi(x)\, f_\epsilon (x)\,
R(x,y)\, \,\{g(y) - g(x)\}
\\
\,& +\, \sum_{x, y\in V} \pi(x)\, \epsilon\, f_\epsilon (x)\,
R(x,y)\,\frac{1}{\epsilon} \big(e^{\epsilon [H(y) -H(x)]} -1\big)
\,\{g(y) - g(x)\}
\end{aligned}
\end{equation}
for every function $g\colon V\to \bb R$.  The first term on the
right-hand side is equal to $2\, \< (-\ms L^s) H, g\>_{\pi}$, 
the second one is equal to $\< f_\epsilon , \ms Lg\>_\pi$, and the
last term vanishes as $\epsilon\to 0$ because $\epsilon f_\epsilon(x)
\to 0$, while $\epsilon^{-1} (\exp\{\epsilon [H(y) -H(x)]\} -1\,)$ is
bounded. 

For each $x\in V$, let $g_x$ be the solution $\ms L g_x = h_x$, where
$h_x\colon V \to \bb R$ is the function given by $h_x(x) = (1-\pi_x)$,
$h_x(y) = -\pi_x$ for $y\neq x$. Since $h_x$ has mean zero with
respect to $\pi$, there is a solution to this equation. Inserting the
function $g_x$ in the previous displayed equation yields that
\begin{equation*}
0\, =\, 2\, \< (-\ms L^s) H, g_x\>_{\pi} \,+\,
\< f_\epsilon , h_x \>_\pi
\,+\;  O(\epsilon) 
\end{equation*}
for every $x\in V$. By definition of $h_x$,
\begin{equation*}
\< f_\epsilon , h_x \>_\pi \,=\, \pi_x \, f_\epsilon(x)
\end{equation*}
because $f_\epsilon$ has mean zero with respect to $\pi$. Therefore,
\begin{equation*}
\pi_x \, f_\epsilon (x) \,=\,
2\, \< \ms L^s H, g_x\>_{\pi} \,+\,  O(\epsilon) 
\end{equation*}
for all $x\in V$, so that $\lim_{\epsilon\to 0} f_\epsilon (x) = 2 \,\pi_x^{-1}\, \<\ms
L^s H, g_x\>_{\pi} =: f (x)$. Moreover, letting $\epsilon\to 0$ in
\eqref{ap-05},
\begin{equation}
\label{ap-06}
2\, \< \ms L^s H, g\>_{\pi} \, =\, \< f , \ms Lg\>_\pi
\end{equation}
for all $g\colon V\to \bb R$.

By \cite[Lemma A.3]{l-gamma},
\begin{equation}
\label{ap-03}
\ms I(\pi_\epsilon) \,=\, \sum_{x,y\in V} \pi_\epsilon(x)\, R(x,y)\,
\big[\, 1 \,-\, e^{\epsilon \, [H(y) \,-\, H(x)]}\,\big]\;.
\end{equation}
Replace $\pi_\epsilon(x)$ by $\pi(x) \, (1 + \epsilon f_\epsilon(x))$,
and expand in $\epsilon$ to obtain that
\begin{equation*}
\begin{aligned}
\frac{1}{\epsilon^2}\, 
\ms I(\pi_\epsilon) \, =\,
& -\, \frac{1}{\epsilon}\, \sum_{x,y\in V} \pi (x)\, R(x,y)\,
[H(y) \,-\, H(x)]
\\
& -\, \frac{1}{2}\, \sum_{x,y\in V} \pi (x)\, R(x,y)\,
[H(y) \,-\, H(x)]^2
\\
& -\, \sum_{x,y\in V} \pi(x) \, f_\epsilon(x)\, R(x,y)\,
[H(y) \,-\, H(x)] \,+\, O(\epsilon) \;.
\end{aligned}
\end{equation*}
The first line vanishes because $\pi$ is the stationary state. The
second one is equal to $\< \ms L^s H, H\>_\pi$. Since $f_\epsilon$
converges to $f$, by \eqref{ap-06}, the third sum converges to $-\, \<
f , \ms LH \>_\pi = 2\, \< (-\ms L^s) H, H\>_\pi $. This completes the
proof of the lemma.
\end{proof}

\subsection*{Acknowledgement}

S. K. has been supported by the KIAS Individual Grant (HP095101) at
the Korea Institute for Advanced Study. S. K. would like to thank IMPA
(Rio de Janeiro) for the warm hospitality during his stay in July
2024.  C. L. has been partially supported by FAPERJ CNE
E-26/201.117/2021, by CNPq Bolsa de Produtividade em Pesquisa PQ
305779/2022-2.

\end{document}